\documentclass{amsart}
\usepackage{amsthm}
\usepackage{lmodern}
\usepackage[english]{babel}
\usepackage{amsmath}
\usepackage{amssymb}
\usepackage{mathptmx} 
\usepackage{color}
\usepackage{graphicx}
\usepackage{caption}
\usepackage{subcaption}
\usepackage{float}
\usepackage[all,cmtip]{xy}
\usepackage{bm}
\usepackage[a4paper, left=3cm, right=2cm]{geometry}
\usepackage{pinlabel}
\usepackage{enumitem}
\usepackage{wrapfig}
\newtheorem{theorem}{\bf Theorem}[section]

\newtheorem{definition}[theorem]{\bf Definition}
\newtheorem{corollary}[theorem]{\bf Corollary}
\newtheorem{proposition}[theorem]{\bf Proposition}

\newtheorem{example}[theorem]{\bf Example}

\newcommand{\rme}{\mathrm{e}}
\newcommand{\rmi}{\mathrm{i}}
\newcommand{\rmd}{\mathrm{d}}

\newcommand{\MN}{\mathcal{M}\hspace{-0.15cm}\mathcal{N}}

\begin{document}

%%%% Article title to be placed here
\title{Saddle point braids of braided fibrations and pseudo-fibrations}

\author{%%%% Author details
Benjamin Bode and Mikami Hirasawa}
\date{}

\address{Instituto de Ciencias Matemáticas (ICMAT), Consejo Superior de Investigaciones Científicas (CSIC), C/ Nicolás Cabrera, 13-15, Campus Cantoblanco, UAM, 28049 Madrid, Spain}
\email{benjamin.bode@icmat.es}
%%%%%%%%% Insert author address here
%\address{Osaka
%ben.bode.2013@my.bristol.ac.uk
%}

\address{Department of Mathematics, Nagoya Institute of Technology, Showa-Ku, Nagoya City, Aichi, 466-8555, Japan}
\email{hirasawa.mikami@nitech.ac.jp}

%%%% Subject entries to be placed here %%%%
%\subject{Mathematical Physics, Topology}

%%%% Keyword entries to be placed here %%%%
%\keywords{braid group action, }
%%% good choices?

%%%% Insert corresponding author and its email address}
%\corres{Benjamin Bode, Mark Dennis\\
%\email{benjamin.bode@bristol.ac.uk}, \email{mark.dennis@bristol.ac.uk}}
%BB: wrote Benjamin instead of Ben
%%% mrd: final final version, from now on all edits require comment

%%%% Abstract text to be placed here %%%%%%%%%%%%
\maketitle
\begin{abstract}
Let $g_t$ be a loop in the space of monic complex polynomials in one variable of fixed degree $n$. If the roots of $g_t$ are distinct for all $t$, they form a braid $B_1$ on $n$ strands. Likewise, if the critical points of $g_t$ are distinct for all $t$, they form a braid $B_2$ on $n-1$ strands. In this paper we study the relationship between $B_1$ and $B_2$.
Composing the polynomials $g_t$ with the argument map defines a pseudo-fibration map on the complement of the closure of $B_1$ in $\mathbb{C}\times S^1$, whose critical points lie on $B_2$.
We prove that for $B_1$ a T-homogeneous braid and $B_2$ the trivial braid this map can be taken to be a fibration map. In the case of homogeneous braids we present a visualisation of this fact.
Our work implies that for every pair of links $L_1$ and $L_2$ there is a mixed polynomial $f:\mathbb{C}^2\to\mathbb{C}$ in complex variables $u$, $v$ and the complex conjugate $\overline{v}$ such that both $f$ and the derivative $f_u$ have a weakly isolated singularity at the origin with $L_1$ as the link of the singularity of $f$ and $L_2$ as a sublink of the link of the singularity of $f_u$. 
\end{abstract}
%%%%%%%%%%%%%%%%%%%%%%%%%%%
%
%%%%%%%%%% Insert the texts which can accomdate on firstpage in the tag "fmtext" %%%%%

\section{Introduction}\label{sec:intro}
%%%% Insert A head here
The braid group on $n$ strands can be defined as the fundamental group of the space of monic, complex polynomials in one variable and of degree $n$. In this way, braids offer a close connection between algebraic geometry and topology. In this paper we illustrate how this connection is beneficial in both directions. We use topological arguments to prove that there exist polynomial maps with singularities and prescribed topological properties and we use insights on polynomial maps to obtain visualisations of certain topological phenomena. Our hope is that we may encourage further research interactions between the two areas.

One well-known intersection of the two areas is the study of links of isolated singularities of polynomial maps. Let $f:=(f_1,f_2):\mathbb{R}^4\to\mathbb{R}^2$ be a polynomial map in variables $x_1$, $x_2$, $x_3$ and $x_4$ that satisfies $f(0,0,0,0)=(0,0)$ and $\tfrac{\partial f_i}{\partial x_j}(0,0,0,0)=0$ for all $i=1,2$, $j=1,2,3,4$. A critical point of $f$ is a point $x$ in $\mathbb{R}^4$ where the gradient matrix $(\tfrac{\partial f_i}{\partial x_j}(x))_{i,j}$ does not have full rank. We denote the set of critical points of $f$ by $\Sigma_f$ and the set of zeros of $f$, i.e., all $x\in\mathbb{R}^4$ with $f(x)=(0,0)$, by $V_f$. In particular, we have $(0,0,0,0)\in \Sigma_f\cap V_f$. We say that $f$ has a \textbf{weakly isolated singularity} at the origin if there is a neighbourhood $U$ of the origin in $\mathbb{R}^4$ such that $U\cap \Sigma_f\cap V_f=\{(0,0,0,0)\}$.

Let $S_{\rho}^3$ denote the Euclidean 3-sphere of radius $\rho$, centered at the origin. If $f$ has a weakly isolated singularity, then the intersection of $V_f$ and $S_{\rho}^3$ results in a closed 1-dimensional submanifold, a link, if $\rho$ is chosen sufficiently small. Furthermore, the link type of $L_f:=V_f\cap S_{\rho}^3$, that is, its ambient isotopy class in $S^3$, is independent of the radius $\rho$, as long as $\rho$ is sufficiently small. This link is thus a topological property of the singularity. We call $L_f$ the link of the singularity.

The name \textit{weakly} isolated singularity is justified in the sense that it does not impose any restrictions on the link types that arise in this way. Akbulut and King proved that every link is the link of a weakly isolated singularity of some polynomial map \cite{ak}. However, there is also a stronger notion of isolation. We say that the origin is an \textbf{isolated singularity} of $f$ if there is a neighbourhood $U$ of the origin in $\mathbb{R}^4$ such that $U\cap \Sigma_f=\{(0,0,0,0)\}$. Naturally, isolated singularities are weakly isolated. A link type is called \textbf{real algebraic} if it arises as the link of an isolated singularity of some real polynomial map.

The definitions above generalise to other dimensions as well as to complex polynomial maps. While links of isolated singularities of complex plane curves $f:\mathbb{C}^2\to\mathbb{C}$ are completely classified, it is not known which links are real algebraic.

A link $L$ in $S^3$ is called \textbf{fibered} if it is the binding of an open book decomposition of $S^3$. In other words, there is a fibration map $\varphi:S^3\backslash L\to S^1$ with a specified behaviour on a tubular neighbourhood of $L$. The term \textit{fibration} means that $\varphi$ has no critical points, i.e., no points $x\in S^3\backslash L$, where the directional derivatives of $\varphi$ in any three linearly independent directions are all 0. The required behaviour of $\varphi$ on the tubular neighbourhood $N(L)$ of $L$ is as follows. Every connected component of $N(L)$ is an open solid torus. Removing $L$ leaves $S^1\times(D\backslash \{0\})$, where $D$ denotes the open (unit) disk in $\mathbb{C}$. On $S^1\times(D\backslash \{0\})$ we require that $\varphi(t,z)=\arg(z)$, where $\arg$ denotes the argument map that sends a non-zero complex number $z$ to $\tfrac{z}{|z|}$. 

The set of fibered links is a promising candidate for the still unknown set of real algebraic links. Milnor proved that all real algebraic links are fibered \cite{milnor}, while Benedetti and Shiota conjecture that the two sets of links are identical \cite{benedetti}.

In the last couple of years several constructions of polynomials with isolated singularities have been proposed \cite{bode:real, bode:satellite, bode:thomo}. These have not resulted in a proof of Benedetti's and Shiota's conjecture, but we now have several infinite families of fibered links that are known to be real algebraic. All of these constructions are based on braids and produce so-called \textbf{semiholomorphic polynomials}. This means that when the constructed function $f:\mathbb{R}^4\to\mathbb{R}^2$ is written as a \textbf{mixed polynomial} $f:\mathbb{C}^2\to\mathbb{C}$ in complex variables $u$, $v$ and their complex conjugates $\bar{u}$ and $\bar{v}$ (which is clearly possible for any real polynomial map), then $f$ is holomorphic with respect to $u$, i.e., $\tfrac{\partial f}{\partial\bar{u}}=0$. From now on we write $f_u$ for $\tfrac{\partial f}{\partial u}$.

With a similar construction we obtained the semiholomorphic version of Akbulut's and King's result: Every link type arises as the link of a weakly isolated singularity of a semiholomorphic polynomial \cite{bode:ak}. One of the main results of this paper is an improvement on this construction. We may perform the construction from \cite{bode:ak} to obtain the desired $f$ and obtain any given link $L_2$ as a sublink of the link of a singularity of $f_u$, i.e., $L_2$ is a subset of the connected components of $L_{f_u}$. We thus have a certain control not only over the topology of $V_f$ but also over the topology of $V_{f_u}$ at the same time.

\begin{theorem}\label{thm:semi}
Let $L_1$ and $L_2$ be links in $S^3$. Then there exists a semiholomorphic polynomial $f:\mathbb{C}^2\to\mathbb{C}$ such that both $f$ and $f_u$ have a weakly isolated singularity at the origin with $L_1$ as the link of the singularity of $f$ and $L_2$ a sublink of the link of the singularity of $f_u$.
\end{theorem}

%Not a topological invariant (neither is being semiholomorphic), $V_{f_u}$ contains all critical points (important to determine if $f$ has isolated singularity), visualise fibrations, illustrates flexibility, i.e., in terms of rigidity semiholomorphic polynomials are not ``half-way'' between real and complex, but much closer to real.

A \textbf{braid} on $n$ strands is a set of disjoint parametric curves
\begin{equation}\label{eq:braidpara}
\bigcup_{t\in[0,2\pi]}\bigcup_{j=1}^n(z_j(t),t)\subset\mathbb{C}\times [0,2\pi],
\end{equation}
where $z_j(t):[0,2\pi]\to\mathbb{C}$ are smooth functions with $z_i(t)\neq z_j(t)$ for all $t\in[0,2\pi]$ and all $i\neq j$. Furthermore, for every index $i$ there should be an index $j$ with $z_i(0)=z_j(2\pi)$. Since the curves are parametrised by the height coordinate $t$, no strand can loop back and cross itself. The parametrisation induces an orientation on the braid, so that it is positively transverse to the horizontal planes $\mathbb{C}\times\{t\}$ for all $t\in[0,2\pi]$. Occasionally, we might encounter a set of parametric curves $B$ in $\mathbb{C}\times[0,2\pi]$ that are transverse to all of these planes, but that are not parametrised by $t$ itself. In this case, there is a choice of orientation for $B$ so that is a braid, but it might not match the orientation induced by the parametrisation. We could thus consider $B$ as an unoriented braid.  

A braid isotopy is an isotopy of a braid in $\mathbb{C}\times[0,2\pi]$ that fixes the start- and endpoints at $t=0$ and $t=2\pi$, and maintains the braid property throughout the isotopy. Often a braid isotopy class is also called a braid. When we want to emphasize that we are referring to a representative of an isotopy class instead of the isotopy class itself, we usually call the representative a geometric braid.

Fixing the start- and endpoints $z_i(0)$, we obtain a group structure on the set of isotopy classes of braids on $n$ strands, where the group operation is given by concatenation and rescaling of the interval. This way we may interpret the braid group on $n$ strands $\mathbb{B}_n$ as the fundamental group of the configuration space of $n$ distinct, unmarked points in the complex plane, where the chosen set of start- and endpoints $z_i(0)$ corresponds to the basepoint of the loop. By the fundamental theorem of algebra this configuration space is diffeomorphic to the space of monic polynomials in one complex variables of degree $n$ and with distinct roots. Every geometric braid, parametrised as in Eq.~\eqref{eq:braidpara}, corresponds to a loop in this space of polynomials, given by $g_t:\mathbb{C}\to\mathbb{C}$,
\begin{equation}
g_t(u)=\prod_{j=1}^n(u-z_j(t)).
\end{equation}
Conversely, the roots of any loop in this space of polynomials form a geometric braid on $n$ strands.

This paper focuses on a subset of this space of polynomials. We define $\widehat{X}_n$ to be the space of monic polynomials of degree $n$ with distinct roots, distinct critical values and constant term not equal to any of its critical values. If the critical values are distinct, so are the critical points. Since the critical points of a complex polynomial $g$ in one variable $u$ are exactly the roots of $\tfrac{\partial g}{\partial u}$, which is a polynomial of degree $n-1$, it follows that we can associate to every loop $g_t$ in $\widehat{X}_n$ three braids: one braid on $n$ strands that is formed by the roots of $g_t$, one braid that is formed by the critical values and one braid on $n-1$ strands that is formed by the critical points of $g_t$. The relation between the braid that is formed by the roots of a loop in $\widehat{X}_n$ with constant term equal to 0 and the braid that is formed by its critical values was studied in detail in \cite{bode:adicact}.

By holomorphicity each critical point of a polynomial in $\widehat{X}_n$ must be a hyperbolic saddle point. For this reason we also call the braid that is formed by the critical points of $g_t$ the \textbf{saddle point braid} of $g_t$. One major topic of this article is the relation between the braid formed by the roots and the saddle point braid of a given loop of polynomials $g_t$ in $\widehat{X}_n$.

\begin{theorem}\label{thm:saddle}
Let $B$ and $B'$ be braids on $n$ and $n-1$ strands, respectively. Then there is a loop $g_t:\mathbb{C}\to\mathbb{C}$ in $\widehat{X}_n$ such that
\begin{itemize}
\item the roots $\{(u,t)\in\mathbb{C}\times [0,2\pi]|g_t(u)=0\}$ form the braid $B$,
\item the saddle point braid, given by $\{(u,t)\in\mathbb{C}\times [0,2\pi]|\tfrac{\partial g_t}{\partial u}(u)=0\}$, is $B'$.
\end{itemize}
\end{theorem}

At almost every height $t\in[0,2\pi]$ we may order the strands of a braid at that height by the real parts of their complex coordinate. That is, if $\bigcup_{j=1}^n(z_j(t),t)$, then the first strand at height $t=t_*$ is the strand with smallest $\text{Re}(z_j(t_*))$. The second strand has the next smallest value of $\text{Re}(z_j(t_*))$ and so on. The braid group on $n$ strands is generated by the Artin generators $\sigma_i$, $i=1,2,\ldots,n-1$, which correspond to positive half-twists between the $i$th strand and the $i+1$th strand.

\begin{definition}\label{def:homo}
A braid $B=\prod_{j=1}^{\ell}\sigma_{i_j}^{\varepsilon_j}$ on $n$ strands is called \textbf{homogeneous} if 
\begin{enumerate}
\item[i)] for every $k\in\{1,2,\ldots,n-1\}$ there is a $j\in\{1,2,\ldots,\ell\}$ with $i_j=k$,
\item[ii)] for every $j,j'\in\{1,2,\ldots,\ell-1\}$, $i_j=i_{j'}$ implies $\varepsilon_j=\varepsilon_{j'}$.
\end{enumerate}
\end{definition}

In other words, a braid is homogeneous if for all $i\in\{1,2,\ldots,n-1\}$ its word contains the generator $\sigma_i$ if and only if it does not contain its inverse $\sigma_{i}^{-1}$. In the literature (in particular in \cite{rudolph2}) braids are sometimes called homogeneous if they satisfy the second condition above and \textit{strictly homogeneous} if they satisfy both conditions above.

Since the start- and endpoints of any geometric braid $B$ match, we may identify the $t=0$-plane and the $t=2\pi$-plane to obtain a link in $\mathbb{C}\times S^1$. Embedding this open solid torus as an untwisted neighbourhood of a planar circle in $S^3$ results in a well-defined link in $S^3$, the closure of the braid $B$, whose ambient isotopy class in $S^3$ does not depend on the representative of the braid isotopy class of $B$. 

Stallings showed that closures of homogeneous braids are fibered links \cite{stallings}. In fact, Definition~\ref{def:homo} has a generalisation in the form of T-homogeneous braids, which were introduced and shown to close to fibered links by Rudolph \cite{rudolph2}. In Section~\ref{sec:loops} we will review the definition and some properties of T-homogeneous braids. Closures of T-homogeneous braids have also recently been proved to be real algebraic \cite{bode:thomo}.

\begin{theorem}\label{thm:thomosaddle}
Let $B$ be a T-homogeneous braid on $n$ strands. Then there is a loop $g_t:\mathbb{C}\to\mathbb{C}$ in $\widehat{X}_n$ such that
\begin{itemize}
\item the roots $\{(u,t)\in\mathbb{C}\times [0,2\pi]|g_t(u)=0\}$ trace out the braid $B$,
\item $(u,t)\mapsto\arg g_t(u)$ is a fibration of $(\mathbb{C}\times S^1)\backslash B$ over $S^1$,
\item the saddle point braid, given by $\{(u,t)\in\mathbb{C}\times [0,2\pi]|\tfrac{\partial g_t}{\partial u}(u)=0\}$, is the trivial braid on $n-1$ strands.
\end{itemize}
\end{theorem}

The condition that the map $(u,t)\mapsto\arg g_t(u)$ is a fibration means that it does not have any critical points. If the map $(u,t)\mapsto\arg g_t(u)$ for a given loop of polynomials $g_t$ has a finite number of non-degenerate critical points, i.e., it is a circle-valued Morse-map, we call it a pseudo-fibration.

Using a construction similar to the one employed in \cite{bode:thomo} results in the following theorem.

\begin{theorem}\label{thm:thomosemi}
Let $L$ be the closure of a T-homogeneous braid $B$ on $n$ strands. Then there is a semiholomorphic polynomial $f:\mathbb{C}^2\to\mathbb{C}$ such that $f$ has an isolated singularity with link $L$ and $f_u$ has a weakly isolated singularity whose link is the unlink with $n-1$ components.
\end{theorem}

The structure of the rest of the article is as follows. Section~\ref{sec:loops} reviews some properties of loops of polynomials in $\widehat{X}_n$, resulting in a proof of Theorem~\ref{thm:thomosaddle} and Theorem~\ref{thm:saddle}. We also give a new upper bound on the Morse-Novikov number of a link, the minimal number of critical points of a pseudo-fibration. In Section~\ref{sec:visual} we then study the argument map $\arg(g_t)$ of a loop of polynomials $g_t$ in $\widehat{X}_n$ and offer different visualisations of the fibration for homogeneous braids as well as of pseudo-fibrations for general braids. In Section~\ref{sec:sing} we modify the constructions of (weakly) isolated singularities from \cite{bode:ak} and \cite{bode:thomo} to prove Theorem~\ref{thm:semi} and Theorem~\ref{thm:thomosemi}.

\section{Saddle point braids}\label{sec:loops}

The space $\widehat{X}_n$ of monic polynomials of fixed degree $n$ and distinct roots, distinct critical values and constant term different from all critical values is defined in such a way that both the roots and the critical points of a loop in $\widehat{X}_n$ form a (closed) braid in $\mathbb{C}\times S^1$. A parametrisation $z_j(t)$, $j=1,2,\ldots,n$, of the roots can be easily obtained by decomposing the polynomials into their irreducible factors, i.e., $g_t(u)=\prod_{j=1}^n(u-z_j(t))$. Likewise, the critical points of a loop $g_t$, which form the saddle point braid, are given by $\bigcup_{j=1}^{n-1}(c_j(t),t)$ where $\tfrac{\partial g_t}{\partial u}(u)=n\prod_{j=1}^{n-1}(u-c_j(t))$. The relation between these polynomials is then obviously $g_t(u)=\int_0^u \tfrac{\partial g_t}{\partial u}(w)\rmd w$.

We may write $v_j(t)=g_t(c_j(t))$, $j=1,2,\ldots,n-1$, for the critical values of $g_t$. The relation between the braid that is parametrised by the $v_j$ (or more precisely the union of the $v_j$ and $\{0\}\times[0,2\pi]$) and the braid that is formed by the roots of $g_t$ was the object of study in \cite{bode:adicact}. One aspect that continues to be relevant in this paper is that deformations of the braid of critical values lift to deformations in $\widehat{X}_n$ and thus also to deformations of the braid of roots and the saddle point braid.

More precisely, we may write
\begin{equation}
V_n:=\{(v_1,v_2,\ldots,v_{n-1})\in(\mathbb{C}\backslash\{0\})^{n-1}: v_i\neq v_j \text{ if }i\neq j\}/S_{n-1},
\end{equation}
for the space of critical values of polynomials in $\widehat{X}_n$, where the symmetric group $S_{n-1}$ acts on $n-1$-tuples of non-zero complex numbers by permutation. Note that the critical values are non-zero, since polynomials in $\widehat{X}_n$ have distinct roots. We then define
\begin{equation}
\widehat{V}_n:=\{(v,a_0)\in V_n\times \mathbb{C}: a_0\neq v_j \text{ for all }j\}.
\end{equation}
Then by Corollary 2.20 in \cite{bode:braided} the map $\theta_n$ that sends a polynomial in $\widehat{X}_n$ to its set of critical values $v\in V_n$ and its constant term $a_0$ is a covering map of degree $n^{n-1}$. 

Suppose now that $g_t$ is a loop in $\widehat{X}_n$ and let $v_t$ denote the loop of the corresponding critical values. That is, if we denote the projection map $\widehat{V}_n\to V_n$ by $\pi$, we have $v_t=\pi(\theta_n(g_t))$. Any homotopy of $v_t$ in $V_n$ is a braid isotopy that therefore extends to an ambient isotopy and gives a homotopy of $\theta_n(g_t)$ in $\widehat{V}_n$, which then lifts to a homotopy of $g_t$ in $\widehat{X}_n$. By definition of $\widehat{X}_n$ the braid type of the braid formed by the roots and the braid type of the saddle point braid do not change during this homotopy.

The relation between the braid formed by the roots and the braid formed by the critical values is important in the context of fibrations. We say that a geometric braid $B$ parametrised by $\bigcup_{j=1}^n(z_j(t),t)$ is \textbf{P-fibered} if the corresponding loop of polynomials $g_t(u)=\prod_{j=1}^n(u-z_j(t))$ defines an explicit fibration map via $\arg(g):(\mathbb{C}\times S^1)\backslash B\to S^1$, where $g:\mathbb{C}\times S^1\to\mathbb{C}$, $g(u,\rme^{\rmi t})=g_t(u)$. We say that a braid isotopy class is P-fibered if it has a P-fibered representative. Sometimes this condition is weakened to require a P-fibered representative in its conjugacy class in $\mathbb{B}_n$.

The property that $\arg(g)$ is a fibration map means that it does not have any critical points. By \cite{bode:satellite} $(u_*,t_*)\in (\mathbb{C}\times S^1)\backslash B$ is a critical point of $\arg(g)$ if and only if $\tfrac{\partial g}{\partial u}(u_*,t_*)=\tfrac{\partial \arg(g)}{\partial t}(u_*,t_*)=0$. This implies that all critical points must lie on the saddle point braid, i.e., there is some $j\in\{1,2,\ldots,n-1\}$ such that $c_j(t_*)=u_*$, where $c_j(t)$, $j=1,2,\ldots,n-1$, parametrise the saddle point braid. The property of being P-fibered is therefore also equivalent to $\tfrac{\partial\arg(v_j(t))}{\partial t}\neq 0$ for all $j$ and all $t$, which has the geometric interpretation that each critical value $v_j(t)$, interpreted as a curve $(v_j(t),t)$ in $\mathbb{C}\times [0,2\pi]$, has a fixed orientation (clockwise or counter-clockwise) with which it twists around $0\times [0,2\pi]\subset\mathbb{C}\times[0,2\pi]$.

We now define a particular family of P-fibered braids, the T-homogeneous braids, which feature in Theorems~\ref{thm:thomosaddle} and \ref{thm:thomosemi}. They are a generalisation of the homogeneous braids defined in Definition~\ref{def:homo}. 

%Instead of the Artin generators $\sigma_i$ of that definition we use a different set of generators of the braid group.

%P-fibered, critical values, $V_n$, cactus, tree, band generators, T-homogeneous definition

Let $T$ be an embedded tree graph in the complex plane with $n$ vertices. Furthermore, every edge $e$ of $T$ should be associated with a sign $\varepsilon_e\in\{+,-\}$. See Figure~\ref{fig:tree}a) for an example. After a planar isotopy of $T$, we may assume that the vertices of $T$ are the $n$ roots of a complex polynomial in $\widehat{X}_n$. We consider loops in $\widehat{X}_n$ whose basepoint is given by the $n$-tuple of the vertices of $T$. For every edge $e$ there is a loop $g_e^{\varepsilon_e}$ in $\widehat{X}_n$ that exchanges the endpoints of the edge, where the sign of the twist matches $\varepsilon_e$ as in Figure~\ref{fig:tree}b). A braid on $n$ strands is called $T$-homogeneous if it has a representative of the form $g_t=\prod_{j=1}^\ell g_{e_j}^{\varepsilon_{e_j}}$, where for every edge $e$ there is an index $j\in\{1,2,\ldots,\ell\}$ with $e_j=e$. The order in which the twists occur or the number of times each twist occurs is not relevant for this definition. In general, we say that a braid $B$ is T-homogeneous if there is some embedded tree $T$ with choice of signs $\varepsilon_e$, such that $B$ is $T$-homogeneous with respect to $T$ and the chosen signs.

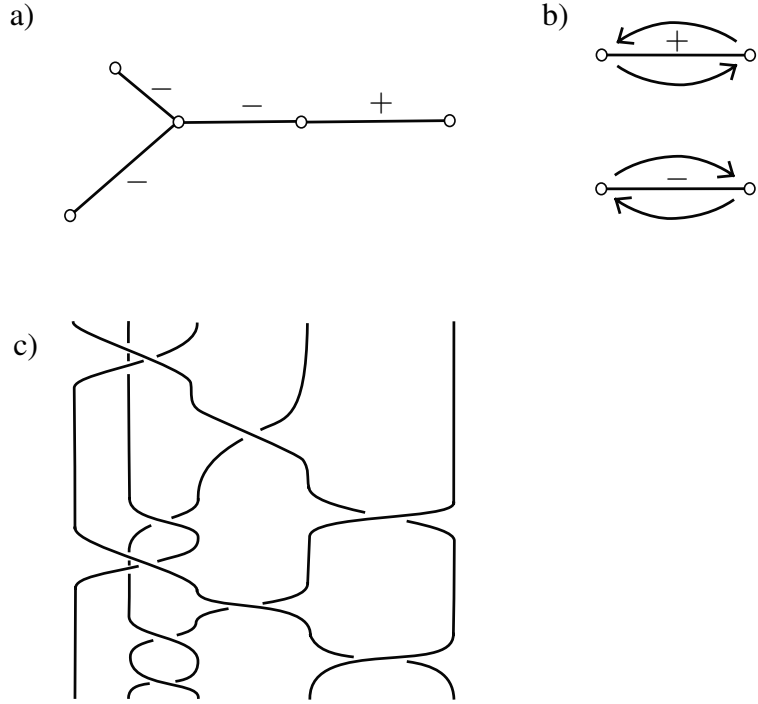
\begin{figure}
\centering
\labellist
\Large
\pinlabel a) at 20 800
\pinlabel b) at 1550 800 
\pinlabel $+$ at 1050 550
\pinlabel $-$ at 680 540
\pinlabel $-$ at 420 590
\pinlabel $-$ at 350 320
\pinlabel $+$ at 1900 730
\pinlabel $-$ at 1900 330
\endlabellist
\includegraphics[height=4.5cm]{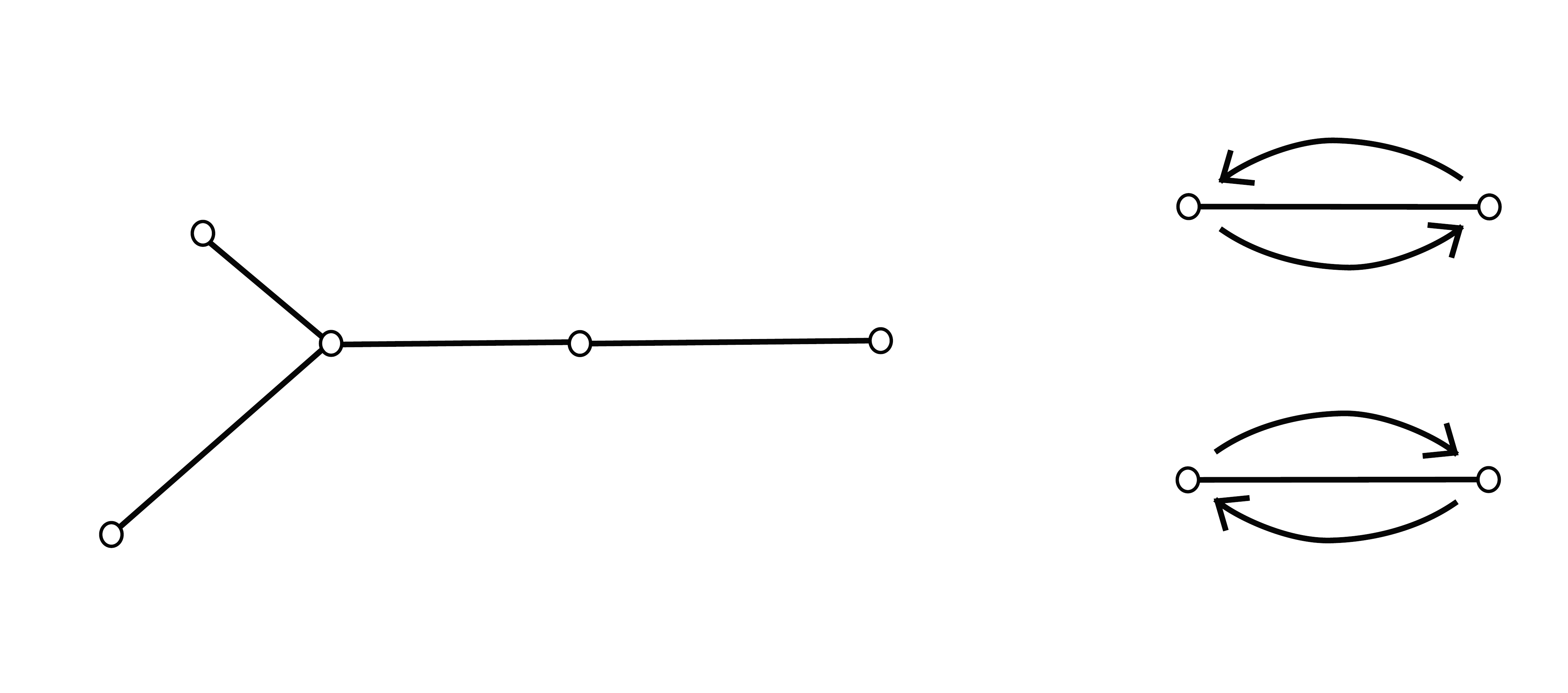}
\labellist
\Large 
\pinlabel c) at 20 1180
\endlabellist
\includegraphics[height=6cm]{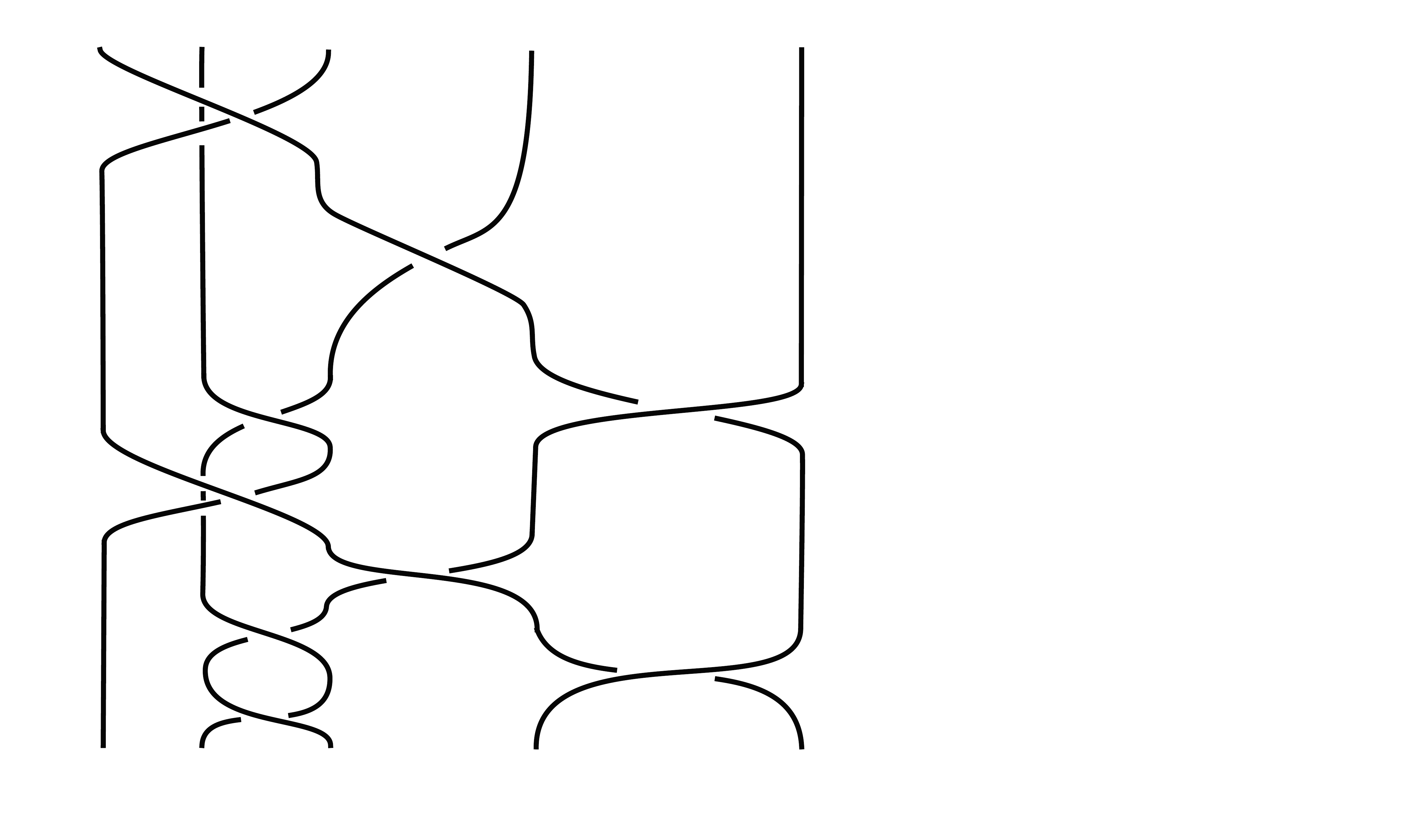}
\caption{a) An embedded tree $T$ in the complex plane. The signs $\varepsilon_e$ are drawn on each edge $e$. b) The loop $g_e$ exchanges the roots on the endpoints of $e$ as in the upper picture if $\varepsilon_e=+1$ and as in the lower picture if $\varepsilon_e=-1$. c) A $T$-homogeneous braid with $T$ as in Subfigure a). \label{fig:tree}}
\end{figure}

Every homogeneous braid $B$ is $T$-homogeneous, if $T$ is the line graph. Numbering the edges from left to right means that the sign $\varepsilon_j$ of the edge $j\in\{1,2,\ldots,n-1\}$ is the unique sign with which $\sigma_j$ appears in the homogeneous braid word $B$.

We now explain how embedded trees arise naturally in the study of complex polynomials. Let $p:\mathbb{C}\to\mathbb{C}$ be a complex polynomial in $\widehat{X}_n$ with critical values $v_k$, $k=1,2,\ldots,n-1$, such that $\arg(v_i)\neq\arg(v_j)$ for all $i\neq j$. Then the argument map $\arg(p)$ induces a singular foliation on $\mathbb{C}$ whose elliptic singularities are the roots of $p$ and whose hyperbolic singularities are the critical points of $p$. Every singular leaf of this foliation has the shape of a cross, consisting of one line connecting two roots of $p$ and one going to the circle boundary $\partial D=S^1$ of $\mathbb{C}\cong D$. The two lines meet in a critical point of $f$. 

%Since $p$ is monic, we have $\lim_{R\to\infty}\arg(p(R\rme^{\rmi\chi}))=n\chi$ for all $$
The $n$th roots of unity divide $\partial D=S^1$ into $n$ arcs. We denote the arcs by $A_j$, $j=1,2,\ldots,n$, increasing the label as we go around the circle clockwise. The choice of arc that is labeled $A_1$ is arbitrary, but it should be the same for all polynomials $p$. Let $c_k$, $k=1,2,\ldots,n-1$ be the critical points of $p$ with $v_k=p(c_k)$. Since the roots of $p$ are distinct, $v_k\neq 0$ for all $k$. We may assume that $0<\arg(v_1)<\arg(v_2)<\ldots<\arg(v_{n-1})<2\pi$.

We may then associate to each critical point $c_k$ (or to each critical value $v_k=p(c_k)$) a transposition $\tau_k\in S_n$, where $S_n$ denotes the permutation group on $n$ elements. As mentioned above, $c_k$ lies on a unique singular leaf of the singular foliation of $D$ induced by $\arg(p)$. This singular leaf has two endpoints on $\partial D=S^1$ that lie on two different arcs $A_i$ and $A_j$, $i\neq j$. Then $\tau_k=(i\ j)$. It turns out that every such list of ordered transpositions $\tau_k$ satisfies $\prod_{k=1}^{n-1}\tau_k=(1\; 2\; \ldots\; n-1 \; n)$. The ordered list $\{\tau_k\}_{k\in\{1,2,\ldots,n-1\}}$ is called the \textbf{cactus} of the polynomial $p$, see \cite{cactus, bode:braided, bode:thomo} for more details. Note that labeling the arcs $A_j$ clockwise is the convention from \cite{bode:thomo}, which is the opposite of that of \cite{bode:braided}.

The definition of $\tau_k$ is illustrated in Figure~\ref{fig:cactus}a). In Figure~\ref{fig:cactus}b) we show the singular leaves of the singular foliation induced by $\arg(p)$, where $p$ is a polynomial of degree 4. The regular leaves can easily be filled in, since they always connect a root to the boundary circle. The cactus of the displayed polynomial is $\tau_1=(1\ 4)$, $\tau_2=(2\ 4)$, $\tau_3=(3\ 4)$.

\begin{figure}[H]
\centering
\labellist
\Large
\pinlabel a) at 100 900
\pinlabel b) at 1200 900
\pinlabel $A_i$ at 580 940
\pinlabel $A_j$ at 500 85
\pinlabel $i$ at 390 940
\pinlabel $j$ at 690 90
\pinlabel $A_1$ at 1280 800
\pinlabel $A_4$ at 1280 250
\pinlabel $A_3$ at 1950 250
\pinlabel $A_2$ at 1950 800
\pinlabel 2 at 1610 1000
\pinlabel 1 at 1140 510
\pinlabel 4 at 1615 50
\pinlabel 3 at 2080 520
\endlabellist
\includegraphics[height=5cm]{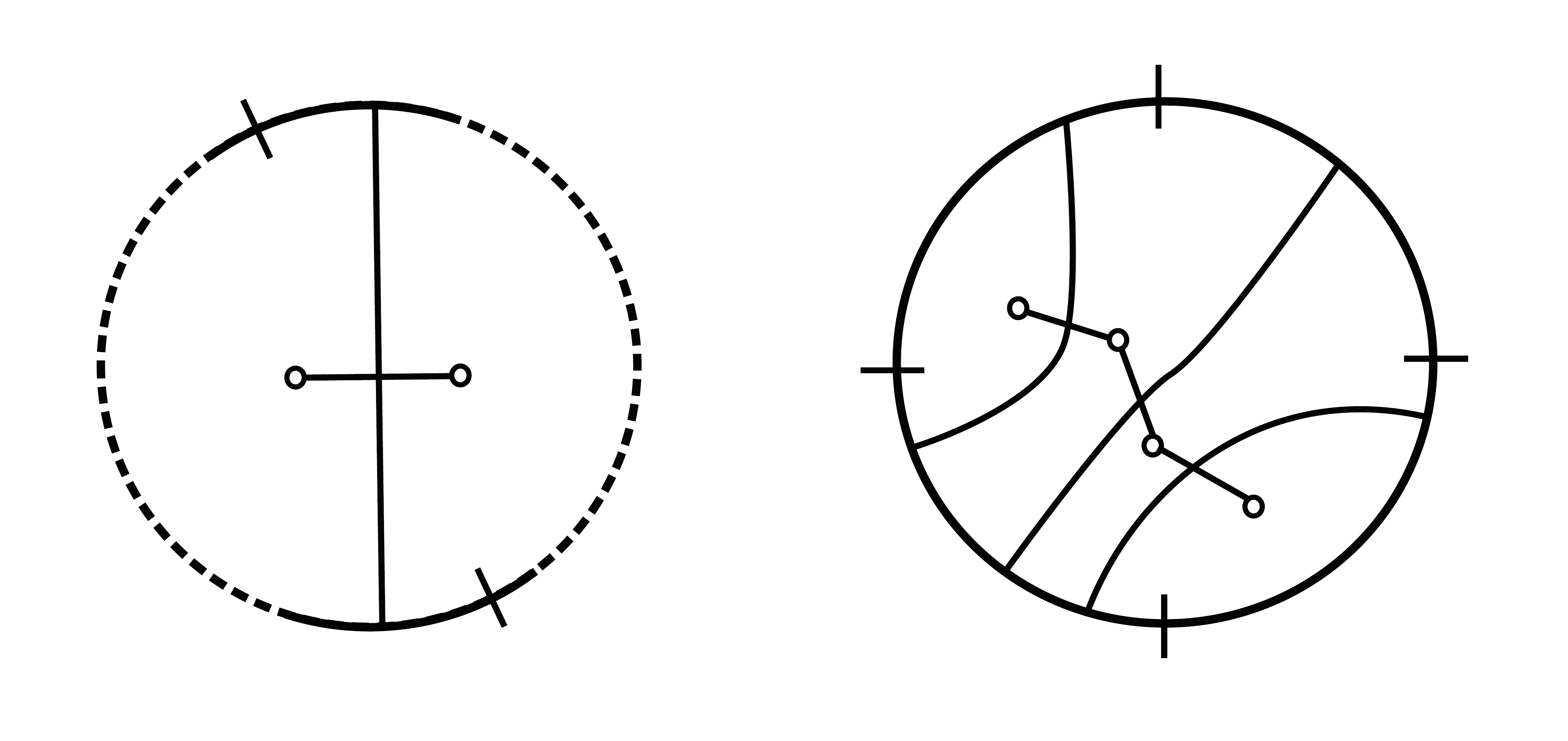}
\caption{Defining the cactus of a polynomial $p$ from the singular foliation induced by $\arg(p)$. a) A singular leaf with a critical point $c_k$ with $\tau_k=(i\ j)$. b) The singular leaves of the singular foliation induced by $\arg(p)$.\label{fig:cactus}}
\end{figure}
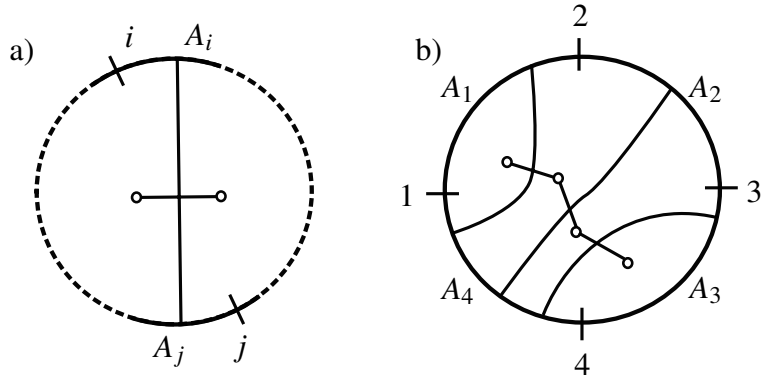

We may also associate to $p$ a planar graph embedded in $\mathbb{C}$, whose vertices are the roots of $p$ and two roots are connected by an edge if and only if they are part of the same singular leaf. This combinatorial structure is essentially equivalent to the cactus of the polynomial $p$. The resulting graph is always a tree and up to planar isotopy every embedded tree arises in this way.

There is a connection between embedded trees and subsets of the BKL-generators $a_{i,j}$, $i,j\in\{1,2,\ldots,n\}$, $i\neq j$, which were introduced by Birman, Ko and Lee in \cite{bkl}, and which also generate $\mathbb{B}_n$. The generator $a_{i,j}$ represents a positive half-twist between the $i$th strand and the $j$th strand such that in the projection the twist lies in front of all strands whose indices are between $i$ and $j$. See Figure~\ref{fig:BKL} for an example. In particular, the Artin generator $\sigma_i$ is given by $a_{i,i+1}$. BKL-generators (or band generators as they are also called) are often used to describe braided surfaces in $\mathbb{C}\times [0,2\pi]$, where each strand corresponds to a disk and each generator $a_{i,j}$ to a half-twisted band between the $i$th and the $j$th disk. In this way, a word in the BKL-generators describes a surface whose boundary is the corresponding braid represented by the same BKL-word.

\begin{figure}[h]
%\labellist
%\Large
%\pinlabel a) at 100 950
%\pinlabel b) at 1250 950
%\endlabellist
\centering
\includegraphics[height=5cm]{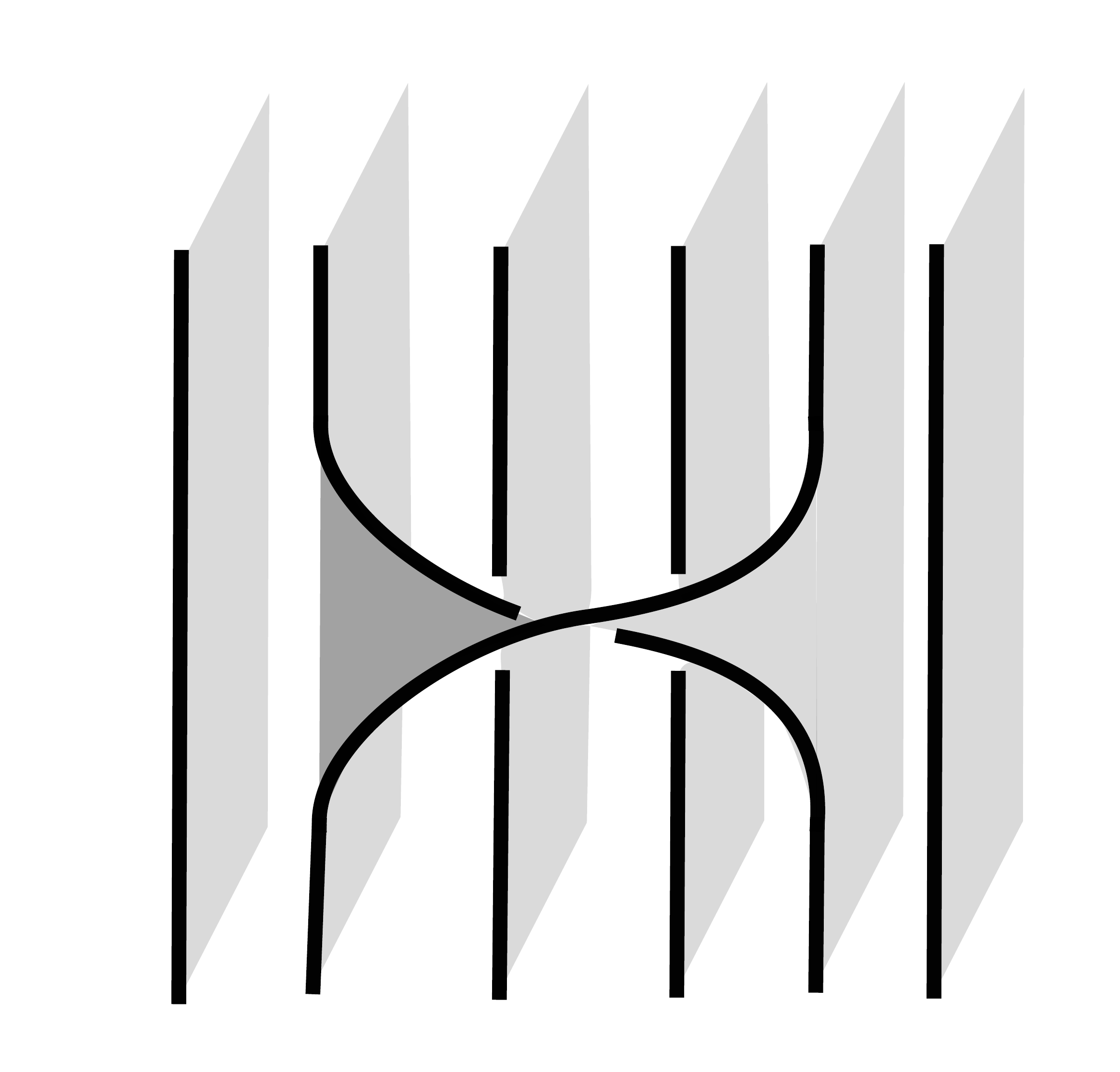}
\caption{The band generator $a_{2,5}$.\label{fig:BKL}}
\end{figure}

For every embedded tree $T$ the twists $g_{e_j}$ around edges $e_j$, $j=1,2,\ldots,n-1$, generate the braid group on $n$ strands, where $n$ is the number of vertices of $T$, see \cite{rudolph2}. After a planar isotopy of $T$ (which does not change the set of corresponding $T$-homogeneous braids) these generators can be realised as a subset of the band generators or BKL generators \cite{rudolph2, bkl}. We call the set $S_T$ of BKL-generators associated to a given embedded tree $T$ the $T$-generators. This means that we can define the set of $T$-homogeneous braids analogously to Definition~\ref{def:homo}. Naturally this definition is equivalent to the one explained above in terms of embedded trees and twists around edges.

\begin{definition}\label{def:thomo}
Let $S_T=\{s_1,s_2,\ldots,s_{n-1}\}$ be the set of $T$-generators of $\mathbb{B}_n$ for some embedded tree $T$. Let $B=\prod_{j=1}^\ell s_{i_j}^{\varepsilon_j}$ be a braid word in the $T$-generators. Then $B$ is a $T$-homogeneous braid word if
\begin{enumerate}
\item[i)] for every $k\in\{1,2,\ldots,n-1\}$ there is a $j\in\{1,2,\ldots,\ell\}$ with $i_j=k$,
\item[ii)] for every $j,j'\in\{1,2,\ldots,\ell-1\}$, $i_j=i_{j'}$ implies $\varepsilon_j=\varepsilon_{j'}$.
\end{enumerate}
We say that a braid $B$ is T-homogeneous if there exists an embedded tree $T$ with $T$-generators $S_T$ such that $B$ can be represented by a $T$-homogeneous braid word.
\end{definition} 

Clearly, Definition~\ref{def:thomo} reduces to Definition~\ref{def:homo} if $T$ is taken to be the line graph and $S_T$ is the set of Artin generators.

In \cite{bode:polynomial} we introduced the inhomogeneity $\beta(B)$ of a braid $B$, a natural number that measures how far away a given braid word $B$ in Artin generators is from being homogeneous. In particular, $\beta(B)=0$ if and only if $B$ is a homogeneous braid.

We may now define for every embedded tree $T$ the $T$-imhomogeneity $\beta_T(B)$ of a braid $B$ as follows. Denote the $T$-generators by $s_1,s_2,\ldots,s_{n-1}$. Then express $B$ as a word in these generators $B=\prod_{j=1}^\ell s_{i_j}^{\varepsilon_j}$. Then we count for each $i\in\{1,2,\ldots,n-1\}$ the number of sign changes of the generator $s_i$ as we traverse the braid word cyclically. We add all these numbers and add 2 for every generator that does not appear in the braid word at all, neither with a positive nor with a negative sign. This is expressed as
\begin{align}\label{def:inhomo}
\beta_T(B)=&\sum_{i=1}^{n-1}|\{j\in\{1,2,\ldots,\ell-1\}:\exists k\in\{1,2,\ldots,\ell-1\} \text{ s.t. }i_j=i_{j+k\text{ mod }\ell}=i,\nonumber\\
& i_{j+m\text{ mod }\ell}\neq i \text{ for all }m<k\text{ and }\varepsilon_j\varepsilon_{j+k\text{ mod }\ell}=-1\}|\nonumber\\
&+2|\{j\in\{1,2,\ldots,n-1\}:\text{ There is no }k\text{ s.t. }i_k=j\}|.
\end{align}
If $T$ is the line graph, the $T$-generators are the usual Artin generators and $\beta_T(B)=\beta(B)$. Note that by definition $\beta_T(B)=0$ if and only $B$ is a $T$-homogeneous braid. We should interpret $\beta_T(B)$ as a property of a braid word. Of course, we may take any braid, inflate its word artificially by inserting arbitrarily many copies of $s_js_j^{-1}$ and thereby make $\beta_T(B)$ arbitrarily large. If we wanted to insist on a topological invariant, we would thus have to take the minimum over all braid words representing the same braid $B$. For practical purposes, it is of course much simpler to consider $\beta_T(B)$ as a function from the set of words in $T$-generators (and their inverses) to the natural numbers.

In order to prove Theorem~\ref{thm:thomosaddle} we need to find a loop of polynomials $g_t$ in $\widehat{X}_n$ whose roots form a given T-homogeneous braid $B$, realised as a P-fibered geometric braid, and such that its saddle point braid is the trivial braid on $n-1$ strands. Loops of polynomials that realise $B$ as a P-fibered geometric braid have been constructed in \cite{rudolph2, bode:braided} and to some extent (for homogeneous braids) already in \cite{rudolph:complex, bode:thesis, survey}. However, these articles do not mention the saddle point braid in this context. We quickly review the main steps in this construction and explain why the corresponding saddle point braid is the trivial braid.

\begin{proposition}\label{prop:key}
Let $T$ be an embedded tree in $\mathbb{C}$ with $n$ vertices and let $B$ be a word in the $T$-generators. Then there is a loop $g_t$ in $\widehat{X}_n$ such that the roots of $g_t$ form the braid $B$, its saddle point braid is the trivial braid on $n-1$ strands and $\arg(g):(\mathbb{C}\times S^1)\backslash B\to S^1$, $\arg(g)(u,\rme^{\rmi t}):=\arg(g_t(u))$, has exactly $\beta_T(B)$ critical points.
\end{proposition}
\begin{proof}
Let $z_1,z_2,\ldots,z_n$ be the vertices of $T$. After a planar isotopy of $T$ we may assume that $T$ is exactly the embedded tree associated to the polynomial $p(u)=\prod_{j=1}^n(u-z_j(t))$ (see for example Theorem 5.3 in \cite{bode:braided}). We now study the loop of polynomials in $\widehat{X}_n$ whose roots form $B$ and that has $p$ as a basepoint.

Each $T$-generator $s_j$ corresponds to a twist $g_{e_j}$ along an edge $e_j$ of $T$. Thus $B$ is a concatenation of twists $g_{e_j}$ along edges $e_j$ of $T$, say $\prod_{k=1}^\ell g_{e_{j_k}}^{\varepsilon_k}$, where the product refers to concatenation of loops in $\widehat{X}_n$. Each twist $g_{e_j}$ corresponds to a very particular motion of the critical values and the constant term in $\widehat{V}_n$. 

Let $v_1,v_2,\ldots,v_{n-1}$ be the critical values of $p$ and let $a_0$ be its constant term. After a small deformation of the embedded graph, we may assume that $\arg(v_i)\neq\arg(v_j)$ if $i\neq j$. After a translation of $T$ in $\mathbb{C}$, which does not affect the critical values or the graph structure, we may further assume that one of the $z_j$ is equal to 0 and so $a_0=0\neq v_j$ for all $j$, so in particular, $p\in\widehat{X}_n$. There is now a one-to-one correspondence between edges of $T$ and critical values of $p$. The twists $g_{e_j}(u)$ can be explicitly realised as loops in $\widehat{X}_n$ via $p(u)-\gamma_j(t)$, where $\gamma_j(t)$ is a loop in $\mathbb{C}$ with basepoint at the origin and the property that it encircles the critical value $v_j$ counterclockwise in an ellipse that does not contain any other critical values $v_i$ with $i\neq j$ as shown in Figure~\ref{fig:braidiso}a). We denote the inverse loop of $\gamma_j(t)$ that encircles $v_j$ in a clockwise direction by $\gamma_j(t)^{-1}$.

Thus we have realised $B$ as the roots of a loop in $\widehat{X}_n$ that is given by $\prod_{k=1}^\ell (p-\gamma_{j_k}(t)^{\varepsilon_k})$, where again the product refers to concatenation of loops in $\widehat{X}_n$. Since the constant term is the only term that depends on $t$, the saddle point braid of the corresponding loop of polynomials is the trivial braid on $n-1$ strands.

The motion of the critical values of a loop $g_{e_j}$ is shown in Figure~\ref{fig:braidiso}b). As in Figure~\ref{fig:braidiso}c) we may deform the loop of critical values and constant term $-\gamma_j(t)$ such that $v_i(t)$ does not depend on $t$ if $i\neq j$ and $v_j$ encircles $\{0\}\times[0,2\pi]$ counterclockwise in an ellipse. Thus the loop of critical values and constant term $\theta_n(\prod_{k=1}^\ell (p-\gamma_{j_k}(t)^{\varepsilon_k}))$ can be deformed in $\widehat{V}_n$ so that in each interval $t\in\left[\tfrac{2\pi (k-1)}{\ell},\tfrac{2\pi k}{\ell}\right]$ the critical values $v_i(t)$ with $i\neq j_k$ do not depend on $t$, while $v_{j_k}$ moves on an ellipse around the origin, going counterclockwise if $\varepsilon_k=1$ and clockwise if $\varepsilon_k=-1$.

We may now deform the loop of critical values slightly to make $\partial_t\arg(v_i(t))$ non-zero for all $t\in\left[\tfrac{2\pi (k-1)}{\ell},\tfrac{2\pi k}{\ell}\right]$ and all $i\neq j_k$. All of these deformations are homotopies in $\widehat{V}_n$ and lift to a homotopy of $\prod_{k=1}^\ell (p-\gamma_{j_k}(t)^{\varepsilon_k})$ in $\widehat{X}_n$. The resulting loop $g_t$ in $\widehat{X}_n$ therefore still has the same braid of roots and the saddle point braid as $\prod_{k=1}^\ell (p-\gamma_{j_k}(t)^{\varepsilon_k})$, that is, $B$ and the trivial braid on $n-1$ strands, respectively.

The corresponding braid of critical values consists of strands that are motions of points on ellipses. We know that critical points of $\arg(g)$ are exactly points where a critical value $v_i(t)$ changes its direction with which it moves on its ellipse, clockwise or counterclockwise. Since the direction of motion on the ellipse is given by the signs $\varepsilon_k$ with $j_k=i$, this number is exactly  
\begin{align}
|\{j\in\{1,2,\ldots,\ell-1\}:&\exists k\in\{1,2,\ldots,\ell-1\} \text{ s.t. }i_j=i_{j+k\text{ mod }\ell}=i,\nonumber\\
& i_{j+m\text{ mod }\ell}\neq i \text{ for all }m<k\text{ and }\varepsilon_j\varepsilon_{j+k\text{ mod }\ell}=-1\}|.
\end{align}

\begin{figure}[H]
\centering
\labellist
\Large
\pinlabel a) at 40 800
\pinlabel b) at 1100 800
\pinlabel c) at 2200 800
\endlabellist
\includegraphics[height=4cm]{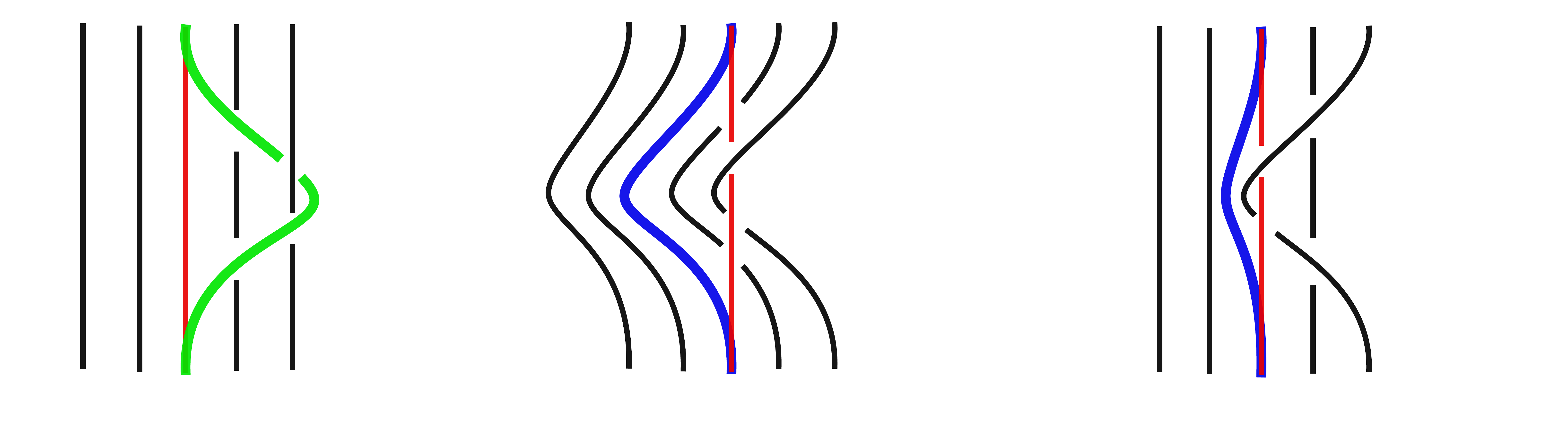}
\includegraphics[height=4.5cm]{inner_loop21}
\caption{a) The critical values $v_j$, $j=1,2,\ldots,n-1$ (in black), the origin $0\in\mathbb{C}$ (in red) and $\gamma_j(t)$ (in thick green), both as curves in $\mathbb{C}\times S^1$ and as motions in $\mathbb{C}$. b) The curves $v_j-\gamma_j(t)$, $j=1,2,\ldots,n-1$ (in black), the origin $0\in\mathbb{C}$ (in red) and the constant term $-\gamma_j(t)$ (in thick blue), both as curves in $\mathbb{C}\times S^1$ and as motions in $\mathbb{C}$. c) A deformation of the curves from Subfigure b) so that all but one critical value become stationary. \label{fig:braidiso}}
\end{figure}

If there is no $k$ with $j_k=i$, i.e., the $T$-generator corresponding to $g_{e_i}$ does not appear in the braid word with any sign, then $v_i(t)$ can be taken to be constant. In order to obtain a Morse function, we deform this stationary strand such that $\arg(v_i(t))$ has exactly two critical points. Thus the number of critical points of $\arg(g)$ is exactly $\beta_T(B)$.
\end{proof}

\begin{proof}[Proof of Theorem~\ref{thm:thomosaddle}]
Let $B$ be a T-homogeneous braid for some embedded tree $T$ with chosen signs. By definition $\beta_T(B)=0$ and so the theorem follows from Proposition~\ref{prop:key}.
\end{proof}

Fibered links in $S^3$ are exactly the bindings of open book decompositions of $S^3$. We say that an open book in $S^3$ is a braided open book if its binding is the closure of a P-fibered braid \cite{bode:braided}. The braid axis can then be thought of as a braid axis for the entire open book, not only for the binding, that is, all fiber surfaces (the pages of the open book) are positioned in a very natural way relative to this braid axis. They are all braided surfaces in the sense of Rudolph \cite{rudolph2}.

An equivalent definition of braided open books in $S^3$ involves simple branched covers $S^3\to S^3$. Montesinos and Morton conjecture that for every fibered link $L$ in $S^3$ there is a simple branched cover $\Pi:S^3\to S^3$ of degree $n$, branched over a link $L_{branch}$ such that $L=\Pi^{-1}(\alpha)$ is the preimage of some braid axis $\alpha$ of $L_{branch}$ and $L_{branch}$ is the unlink on $n-1$ components \cite{morton}.

We showed in \cite{bode:braided} that we can construct a simple branched cover $\Pi:S^3\to S^3$ from any loop of polynomials $g_t$, whose roots form a P-fibered geometric braid. The resulting branch link $L_{branch}$ is exactly the closure of the braid that is formed by the critical values of $g_t$. The proof of Proposition~\ref{prop:key} shows that for closures of T-homogeneous braids the conjecture by Montesinos and Morton is true, since the braid of critical values is (exactly like the saddle point braid) the trivial braid on $n-1$ strands.

In general it is not true that the saddle point braid and the braid of critical values are isotopic, but they must always have the same permutation of strands. So if the saddle point braid is the trivial braid on $n-1$ strands, then the braid of critical values is a pure braid. That is, every critical value ends at $t=2\pi$ in the same position where it starts at $t=0$. Constructing $\Pi$ from such a loop of polynomials $g_t$ gives a branch link $L_{branch}$ with $n-1$ components, but not necessarily the unlink.

Braided open books have also been studied by Rudolph \cite{rudolph2}, who calls the saddle point braid the ``derived bibraid''. The name is justified, since the closure of the saddle point braid, as a link in $S^3$, is transverse to all pages of the open book whose binding is the braid axis and also transverse to all pages of the given braided open book. We may therefore choose orientations for the components of the derived bibraid that turn it into a braid relative to the braid axis and another (in general different) choice of orientation turns it into a generalised braid relative to the fibered link $L$. Singularity theorists might also be interested in Rudolph's calculation of the Milnor number of a fibered link that is the binding of a braided open book in $S^3$ in terms of properties of the derived bibraid \cite{rudolph2}.

The Morse-Novikov number $\MN(L)$ of a link $L$ is a natural number that measures how far a given link is from being fibered. In particular, $\MN(L)=0$ if and only if $L$ is fibered. In \cite{bode:polynomial} we proved the upper bound $\MN(L)\leq \beta(B)$ for all links $L$ and all braids $B$ that close to $L$. We use the discussion above to improve this bound.

The Morse-Novikov number $\MN(L)$ is defined to be the minimal number of critical points of any circle-valued Morse map on $S^3\backslash L$ that displays the usual behaviour of an open book in a tubular neighbourhood of $L$, that is, locally it is given by $\phi:S^1\times(D\backslash\{0\})\to S^1$, $\phi(x,z)=\arg(z)$. We will refer to any such map $\phi$ as a pseudo-fibration. Upper and lower bounds of the Morse-Novikov number in terms of other link invariants have been found in \cite{mikami}, \cite{rudolphMN} and \cite{pajitnov}, but we are not aware of any explicit formula or algorithm that computes it.

\begin{corollary}\label{prop:MNL}
Let $B$ be a braid on $n$ strands whose closure is the link $L$. Then $\MN(L)\leq \min_T \beta_T(B)$, where the minimum is taken over all embedded trees $T$ with $n$ vertices. 
\end{corollary}
\begin{proof}
From Proposition~\ref{prop:key} we have for every embedded tree $T$ with $n$ vertices a loop $g_t$ in $\widehat{X}_n$ whose roots form $B$ and with exactly $\beta_T(B)$ critical points of $\arg(g)$. As in \cite{bode:lemniscate} we can construct from a Morse function $\varphi$ on $S^3\backslash L$, where $L$ is the closure of $B$ with the same number of critical points. In fact, $\varphi$ may be taken to be the argument of a semiholomorphic polynomial map $f:\mathbb{C}^2\to\mathbb{C}$ restricted to $S^3\backslash L$ on the 3-sphere of unit radius. (The polynomial $f$ does not necessarily have a (weakly) isolated singularity and the $S^3_{\rho}\cap V_f$ might be different from $L$ for small radii $\rho$.) Thus $\varphi$ satisfies the desired property on a tubular neighbourhood of $L$ and we have $\MN(L)\leq \beta_T(B)$. Since this holds for any embedded tree $T$, the result follows.
\end{proof}

Note that $\beta_T(B)$ only depends on the set of $T$-generators, not on the embedded tree $T$ per se. Thus the expression $\min_T\beta_T(B)$ refers to the minimum of a finite set of numbers.

\begin{example}
Consider again the example from Figure~\ref{fig:tree}c). We already know that it is $T$-homogeneous for the embedded tree $T$ in Figure~\ref{fig:tree}a). Therefore, $\beta_T(B)=0$ and $\MN(L)=0$. However, expressing the same braid in Artin generators gives $\sigma_2^{-1}\sigma_4\sigma_2^{-1}\sigma_3^{-1}\sigma_1\sigma_2^{-1}\sigma_1^{-1}\sigma_2^{-1}\sigma_4\sigma_3^{-1}\sigma_1\sigma_2^{-1}\sigma_1^{-1}$. We now calculate $\beta_T(B)=\beta(B)$ for the line graph $T$. Every generator appears with a positive sign or a negative sign, so the last sum in Eq.~\eqref{def:inhomo} does not contribute. Furthermore, the generators $\sigma_2$, $\sigma_3$ and $\sigma_4$ all come with a fixed sign. All instances of $\sigma_2$ and $\sigma_3$ are negative, while all instances of $\sigma_4$ are positive. So these strands do not contribute to the count in $\beta_T(B)$. However, the sequence of signs of $\sigma_1$ as we traverse the braid word reads $\{+,-,+,-\}$. So there are three sign changes plus one, since the first entry of this list is different from the last one. Thus the bound from \cite{bode:polynomial} would have given $\MN(L)=0\leq 4=\beta(B)$, while our new improved bound in Proposition~\ref{prop:MNL} gives equality $\MN(L)=0=\min_T \beta_T(B)$.
\end{example}

\begin{proof}[Proof of Theorem~\ref{thm:saddle}]
Consider a parametrisation of the braid $B'$, say
\begin{equation}
\bigcup_{t\in[0,2\pi]}\bigcup_{j=1}^{n-1}(c_j(t),t)\subset\mathbb{C}\times[0,2\pi],
\end{equation} 
with appropriate functions $c_j:[0,2\pi]\to\mathbb{C}$. Then $h_t(u):=n\int_{0}^u\prod_{j=1}^{n-1} (w-c_j(t))\rmd w$ is a loop in the space of monic polynomials of degree $n$ whose critical points form the braid $B'$ in exactly the given parametrisation. After a small deformation of $h_t$ we may assume that its roots are also distinct. The fact that its critical points form $B'$ does not change with this small deformation. 

Since the roots of $h_t$ are distinct for all $t\in[0,2\pi]$, they form a braid on $n$ strands. However, at this stage we do not know what this braid is. We will call it $A$. Now apply the construction outlined in the proof to Proposition~\ref{prop:key} to the braid $A':=A^{-1}B$ and basepoint $g_0:=h_0$ to obtain a loop $g_t$ in $\widehat{X}_n$ whose roots form the braid $A'$ and whose saddle point braid is the trivial braid $e$ on $n-1$ strands. Then the composition of $h_t$ and $g_t$ is a loop in $\widehat{X}_n$ whose roots form the braid $AA'=B$ and whose critical points form the braid $B'e=B'$.
\end{proof}

\section{Visualisations of (pseudo)fibrations}\label{sec:visual}

There are several characterisations of the set of fibered links in terms of link invariants \cite{stallings2, gabai, ni}. However, even if the proof that a certain link is fibered offers a description of the corresponding fiber surface, it is often difficult to visualise how exactly these fibers fill the link complement. For closures of P-fibered braids the whole fibration is in a nice position relative to the unbook in $S^3$ or if we consider the corresponding braid in $\mathbb{C}\times [0,2\pi]$, we might say the fibration is in a nice position relative to the height function $(u,t)\mapsto t$ as in \cite{bode:braided}. This allows for nice visualisations of the fibrations of complements of closures of P-fibered braids.

We present three different methods to visualise such fibrations. In the case of non-fibered links similar tools may be used to visualise pseudo-fibration maps. Before we describe these visualisations we would like to mention that the structure of a link and surfaces that foliate its complement also appear in various physical systems, so that visualisations of this form are of interest beyond pure mathematics. 

Running a constant electric current through a closed wire in the shape of a given knot $K$ induces a magnetic field $B:\mathbb{R}^3\backslash K\to\mathbb{R}^3$, which can be calculated explicitly (numerically) from the Biot-Savart law. It has a circle-valued magnetostatic potential function $\varphi:\mathbb{R}^3\backslash K\to S^1$ with $\nabla \varphi=B$. More details can be found in \cite{gareth}. Under a small assumption on its behaviour near the point at infinity the map $\varphi$ can be completed to a map on $S^3\backslash K$ with the desired behaviour on a tubular neighbourhood of $K$, so that $\varphi$ has at least $\MN(K)$ critical points. The critical points of $\varphi$ are precisely the zeros of the magnetic field $B$. The pages of an open book thus have a physical interpretation as the level sets of a magnetostatic potential function. Note however that both $B$ and $\varphi$ depend heavily on the geometry of $K$. Different, but isotopic, embeddings can lead to very different fields and potentials. In particular, it is not known if it is possible to arrange $K$ in $\mathbb{R}^3$ such that the resulting $\varphi$ has exactly $\MN(K)$ points.

Other interpretations of fibrations appear in the context of liquid crystals, where the knot is a defect line and the fiber surfaces correspond to layers of material along which molecules arrange themselves \cite{randy}. In this context, critical points of a map $\varphi:\mathbb{R}^3\backslash K\to S^1$ correspond to point defects of the liquid crystal configuration.

In singularity theory, the fibrations play an important role via Milnor fibrations, given by the argument of a polynomial map with isolated singularity. The visualisation of fibrations presented in this section can therefore be interpreted as visualisations of Milnor fibrations on the 3-sphere $S^3_{\varepsilon}\backslash V_f$, where the link of the singularity is presented as a braid.

\subsection{Visualisations from explicit functions}

Suppose that $B$ is a P-fibered geometric braid realised via the roots of $g_t$, a loop in $\widehat{X}_n$. If we know the maps $g_t$ and thus $g(u,\rme^{\rmi t}):=g_t(u)$, we can simply plot the level sets of $\arg(g)$ to obtain a nice visualisation.

\begin{figure}[H]
\centering
\labellist
\Large
\pinlabel a) at 40 360
\endlabellist
\includegraphics[height=4cm]{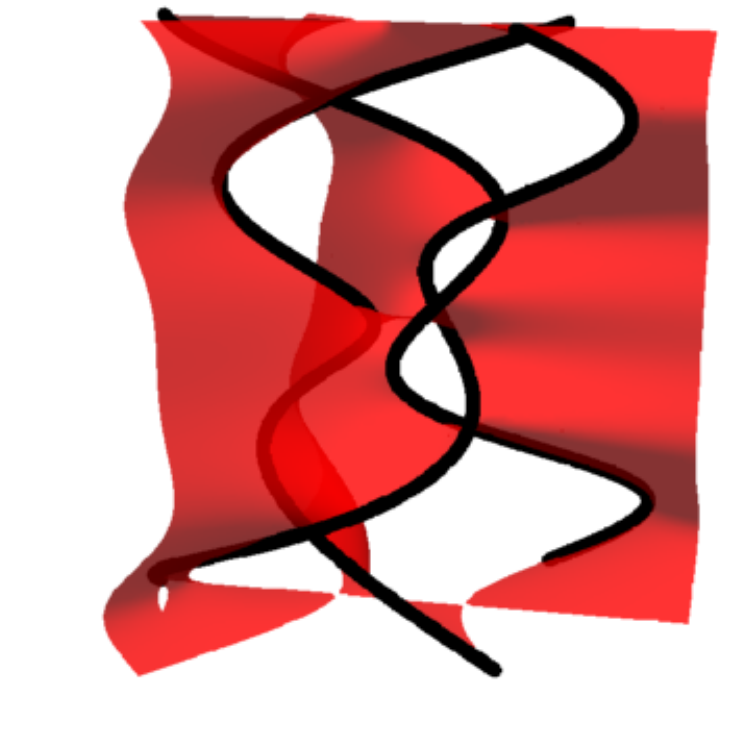}
\labellist
\Large
\pinlabel b) at 40 360
\endlabellist
\includegraphics[height=4cm]{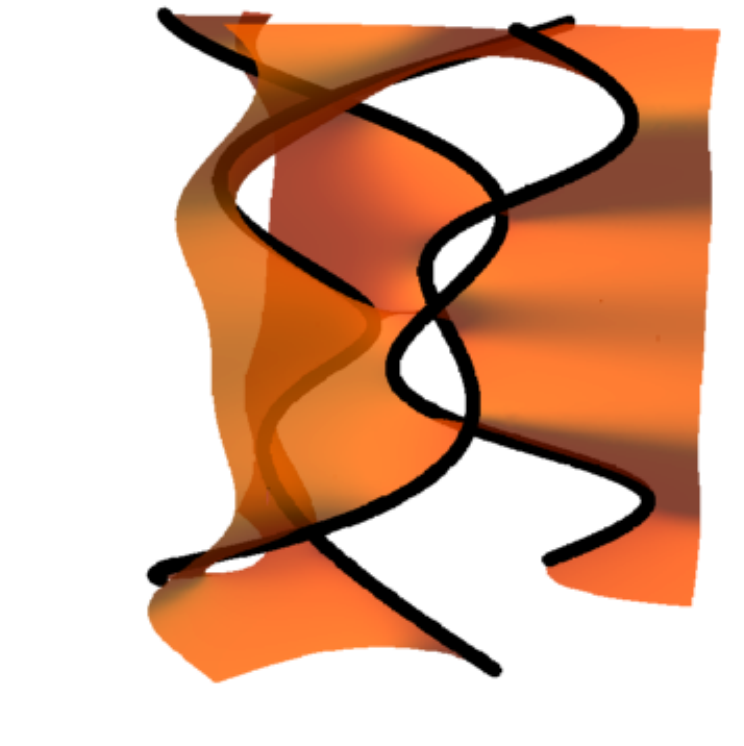}
\labellist
\Large
\pinlabel c) at 40 360
\endlabellist
\includegraphics[height=4cm]{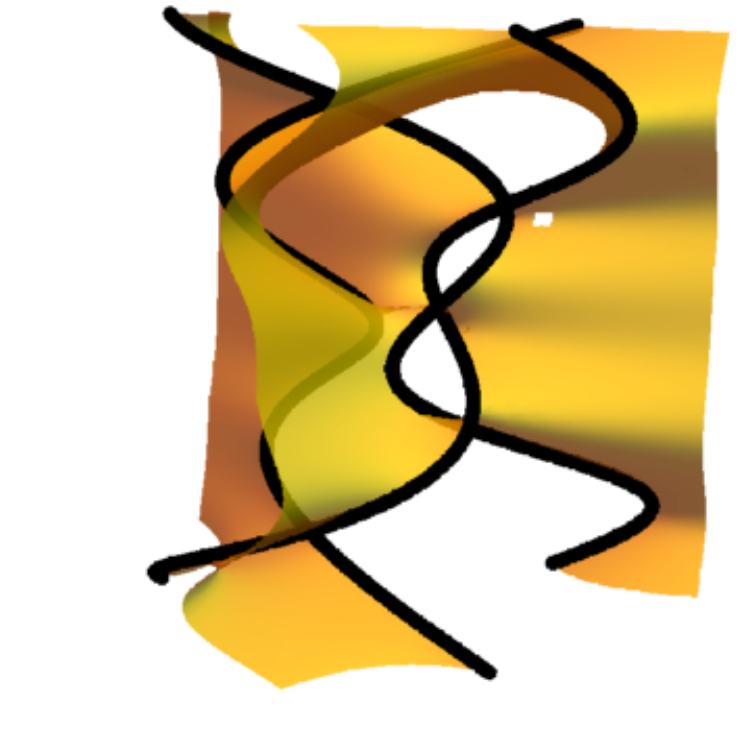}
\labellist
\Large
\pinlabel d) at 40 360
\endlabellist
\includegraphics[height=4cm]{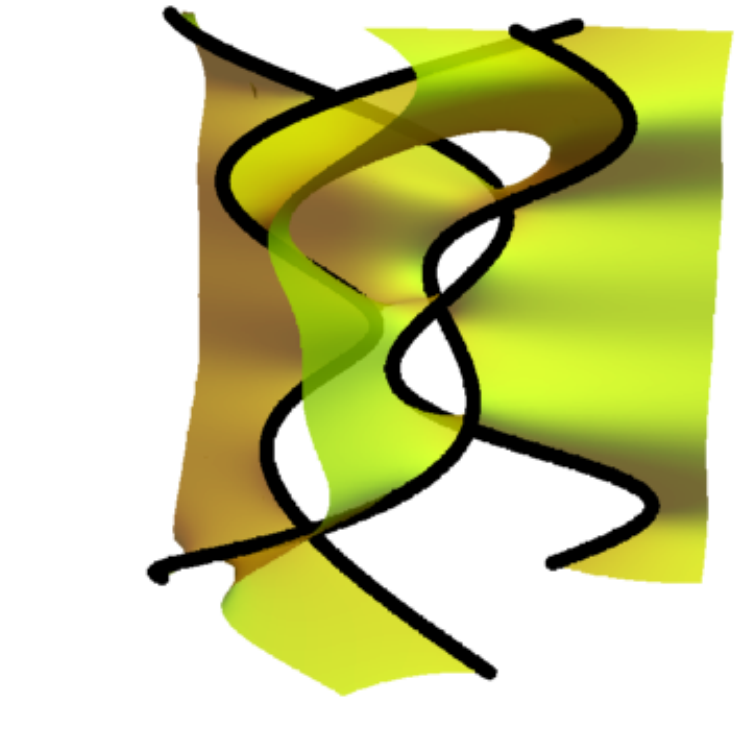}
\labellist
\Large
\pinlabel e) at 40 360
\endlabellist
\includegraphics[height=4cm]{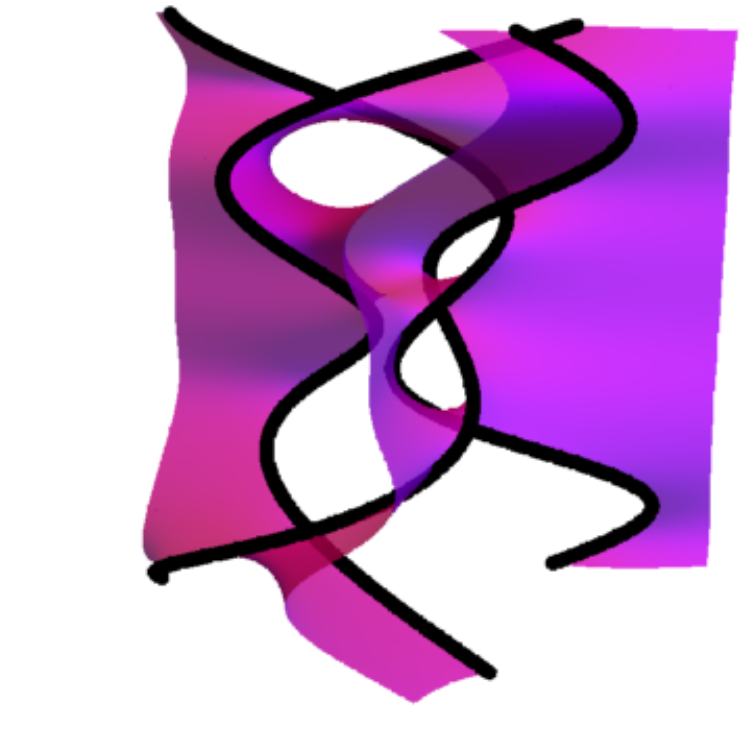}
\labellist
\Large
\pinlabel f) at 40 360
\endlabellist
\includegraphics[height=4cm]{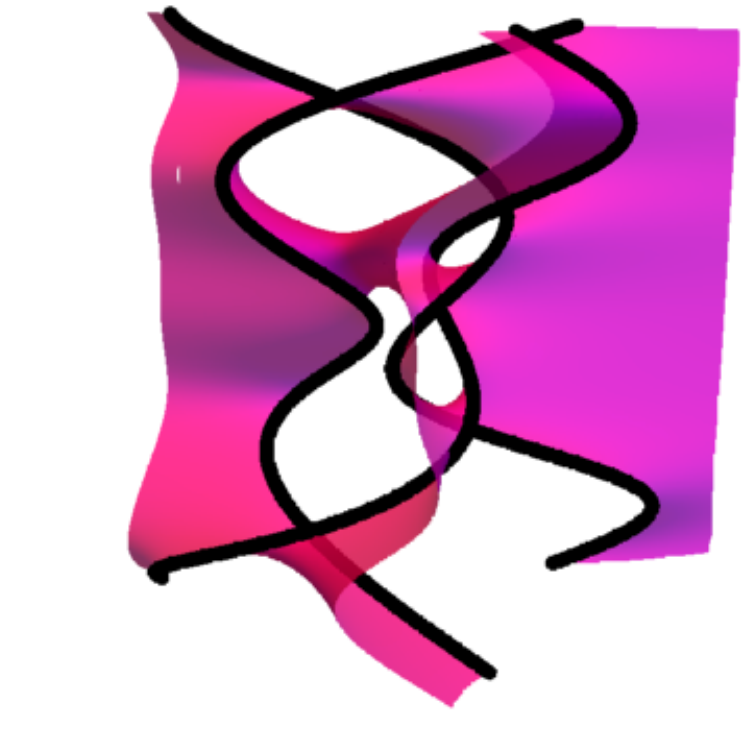}
\labellist
\Large
\pinlabel g) at 40 360
\endlabellist
\includegraphics[height=4cm]{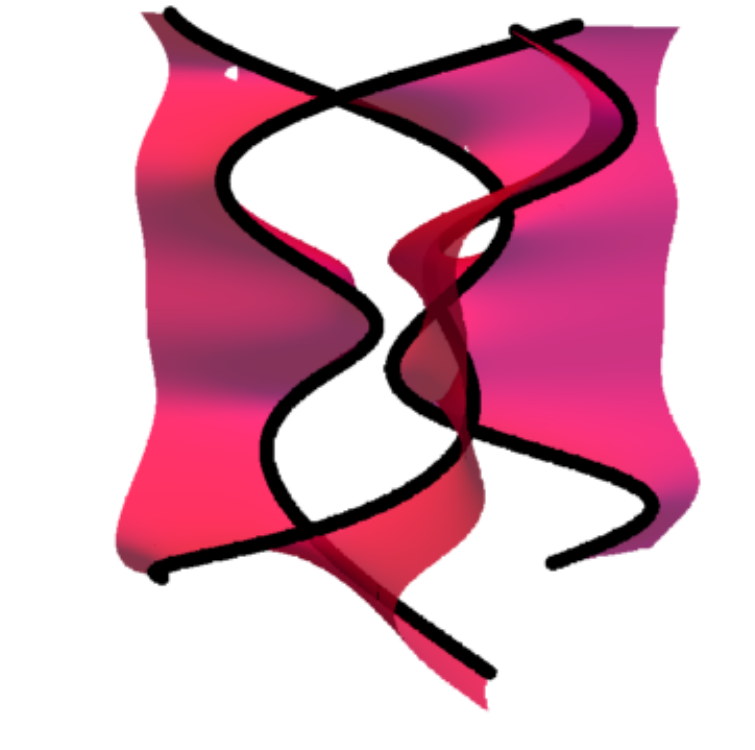}
\labellist
\Large
\pinlabel h) at 40 360
\endlabellist
\includegraphics[height=4cm]{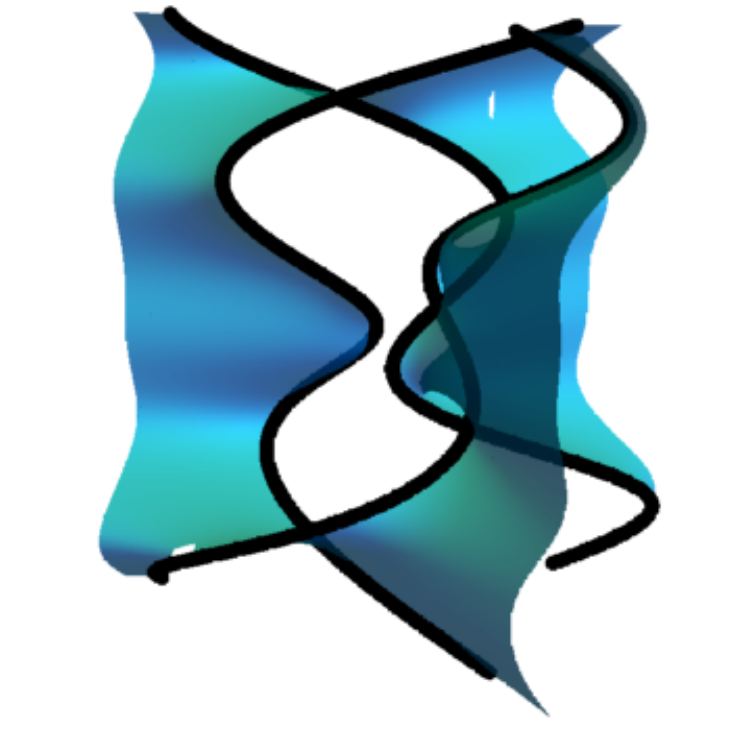}
\labellist
\Large
\pinlabel i) at 40 360
\endlabellist
\includegraphics[height=4cm]{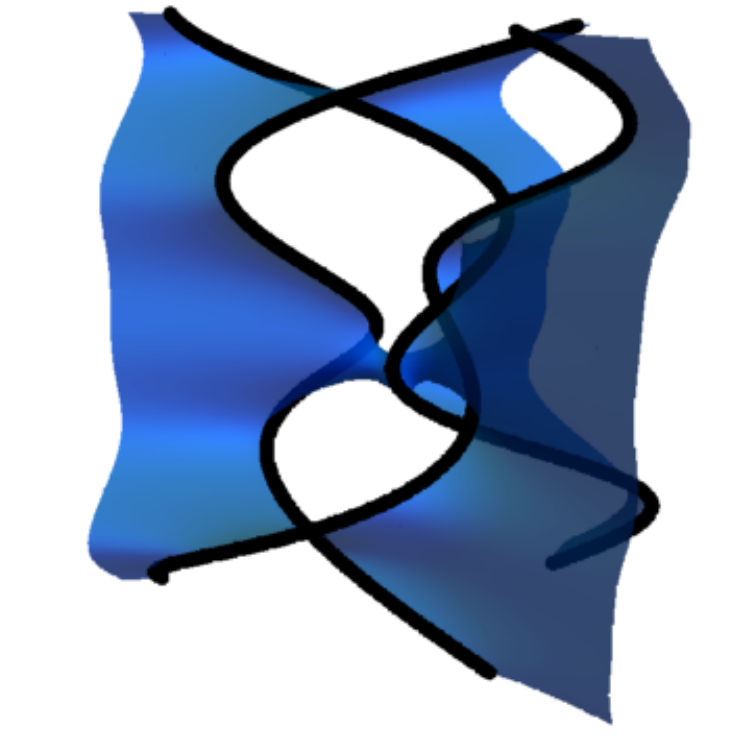}
\labellist
\Large
\pinlabel j) at 40 360
\endlabellist
\includegraphics[height=4cm]{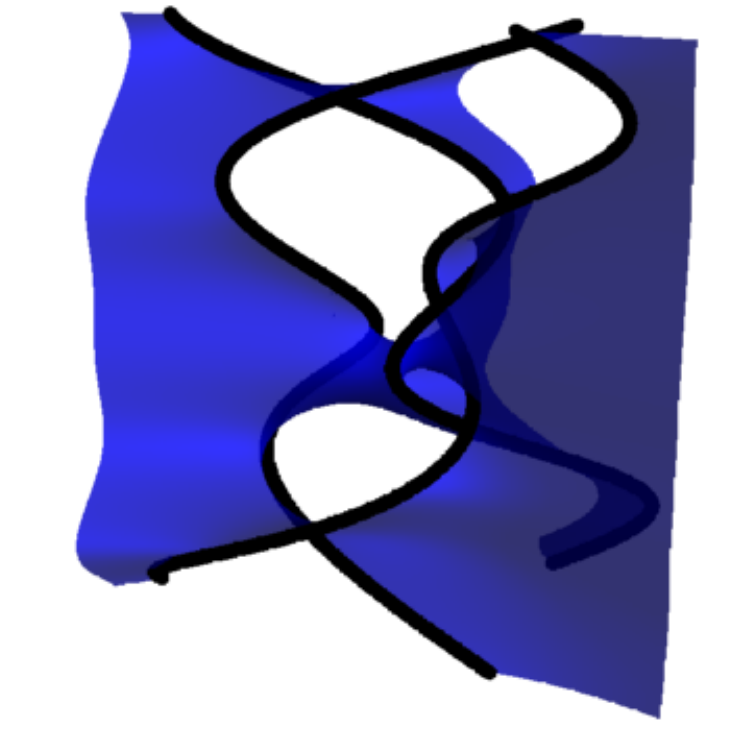}
\labellist
\Large
\pinlabel k) at 40 360
\endlabellist
\includegraphics[height=4cm]{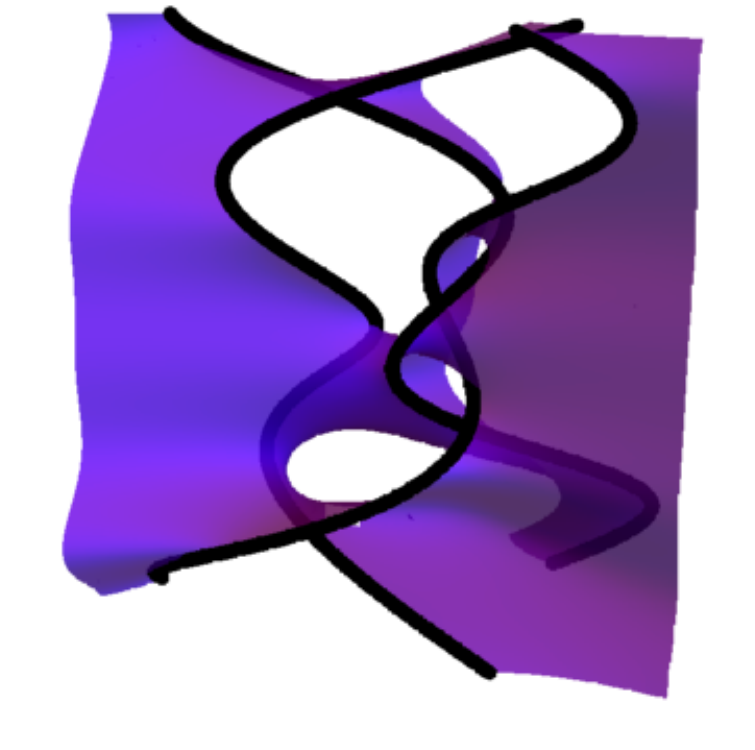}
\labellist
\Large
\pinlabel l) at 40 360
\endlabellist
\includegraphics[height=4cm]{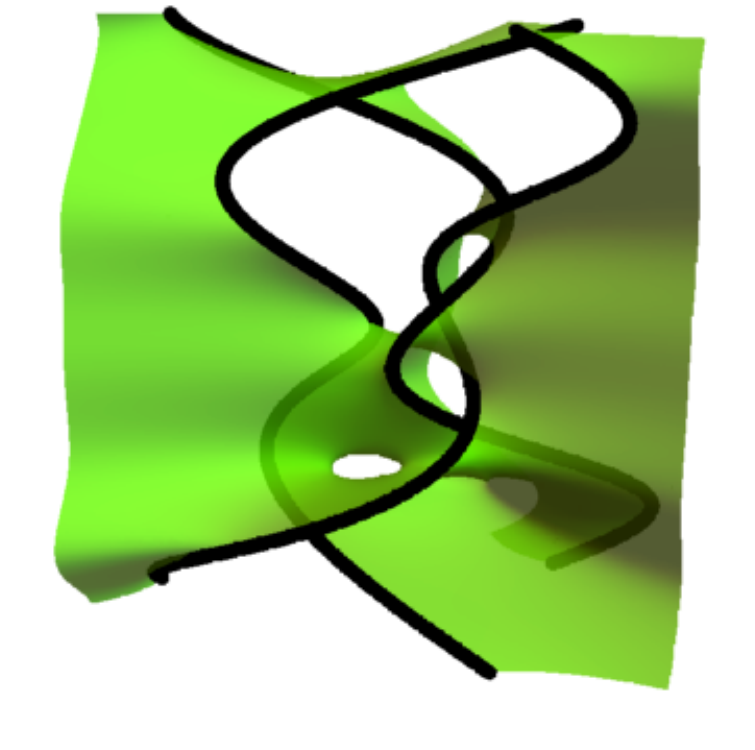}
\labellist
\Large
\pinlabel m) at 40 360
\endlabellist
\includegraphics[height=4cm]{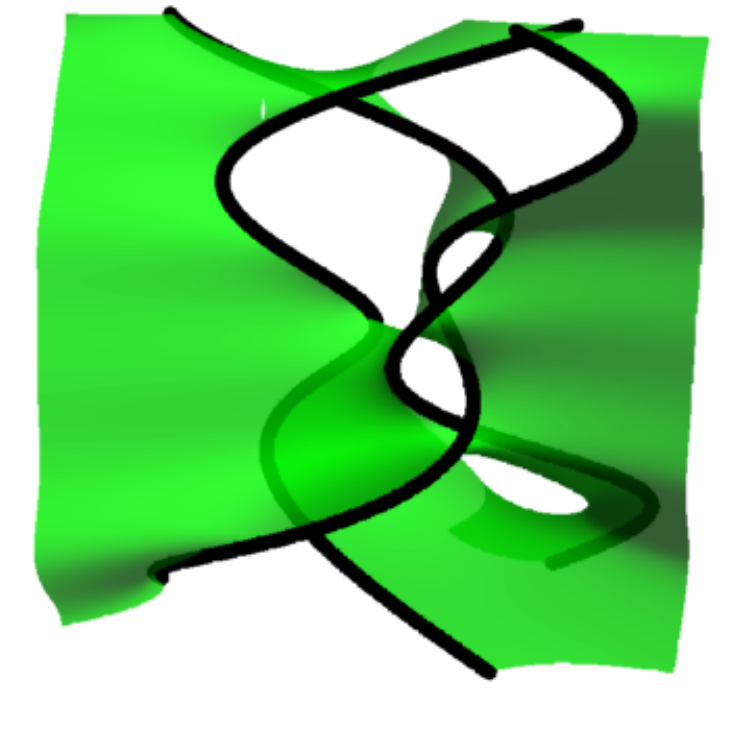}
\labellist
\Large
\pinlabel n) at 40 360
\endlabellist
\includegraphics[height=4cm]{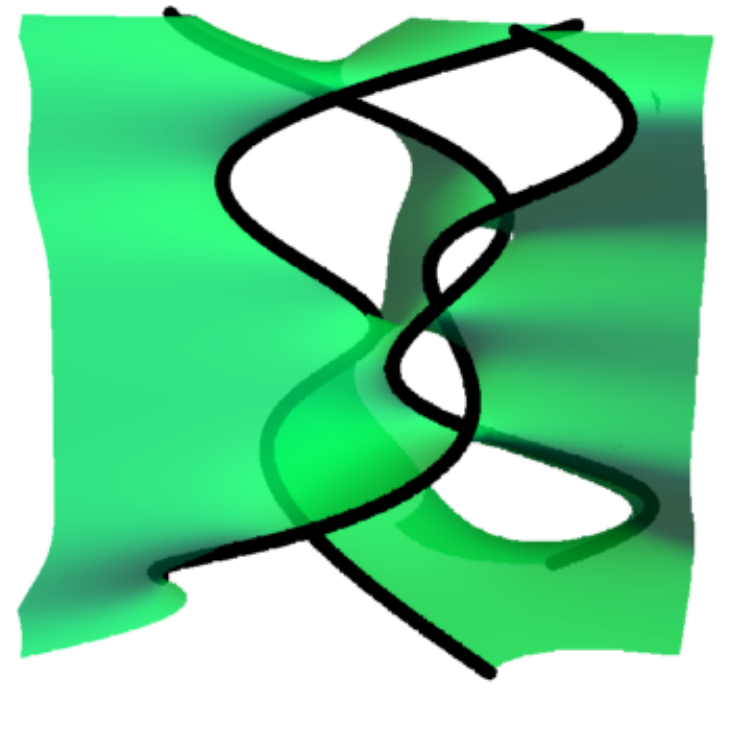}
\labellist
\Large
\pinlabel o) at 40 360
\endlabellist
\includegraphics[height=4cm]{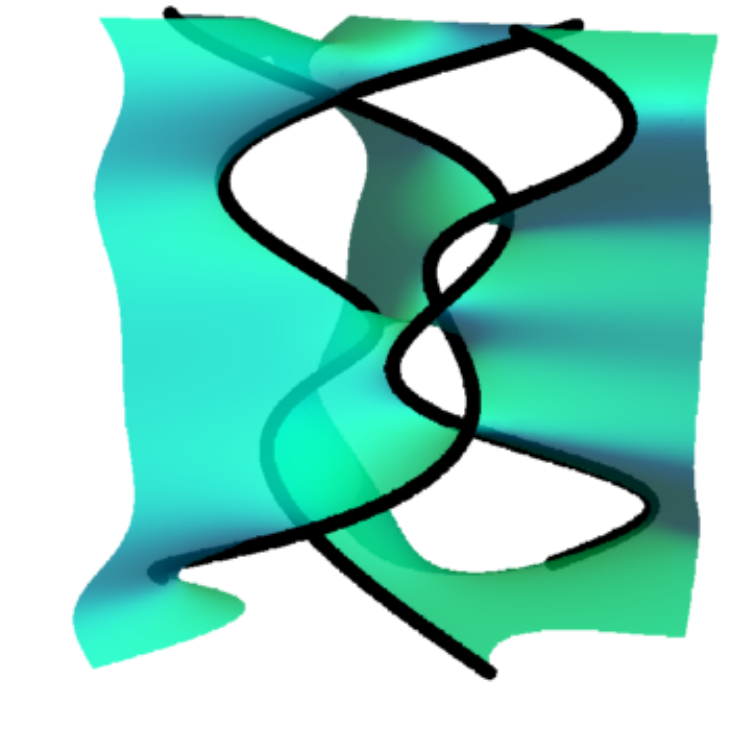}
\caption{Level sets $(\arg(g))^{-1}(\rme^{\rmi \chi})$ of $\arg(g)$. From Subfigure a) to Subfigure o) the value $\chi$ varies from 0 to $\tfrac{28}{15}\pi$. \label{fig:fibers}}
\end{figure}

Often we do not know $g_t$. For example, we know that T-homogeneous braids are P-fibered and therefore have a P-fibered geometric braid in its isotopy class, but do not have an explicit parametrisation of this representative and therefore do not have an explicit description of $g_t$.

However, with methods as in \cite{bode:polynomial, 52} we may find a parametrisation of a representative of any given braid $B$ in terms of trigonometric polynomials. This leads to a visualisation of the corresponding pesudo-fibration. Furthermore, we may use this approach to prove that for some given links $L$ its Morse-Novikov number can be realised by the argument of a polynomial map.

%However, this approach is useful for the visualisation of pseudo-fibrations and proving that for a given link $L$ the Morse-Novikov number can be realised by the argument of a polynomial map. 
We illustrate this with the example of the knot $5_2$, which is the closure of the braid $B=\sigma_1\sigma_2^{3}\sigma_1\sigma_2^{-1}$ and was discussed in more detail in \cite{52}.

As in \cite{bode:polynomial, 52} we find a parametrisation of the braid $B$ in terms of trigonometric polynomials via
\begin{equation}
\bigcup_{t\in[0,2\pi]}\bigcup_{j=1}^3(z_j(t),t)\subset\mathbb{C}\times [0,2\pi]
\end{equation}
with
\begin{align*}
z_j(t)=&-\cos\left(\frac{2(t+2\pi j)}{3}\right)-\frac{3}{4}\cos\left(\frac{5(t+2\pi j)}{3}\right)\\
&-\rmi\left(\sin\left(\frac{4(t+2\pi j)}{3}\right)+\frac{1}{2}\sin\left(\frac{t+2\pi j}{3}\right)\right). 
\end{align*}

Then the corresponding loop of polynomials is as usual $g_t(u):=\prod_{j=1}^3(u-z_j(t))$. Since the knot $5_2$ is not fibered, the resulting argument map $\arg(g)$, $g(u,\rme^{\rmi t}):=g_t(u)$, must have critical points. We have $\MN(5_2)=2$.

Since we know the function $g_t$ we can plot level sets of $\arg(g)$, see Figure~\ref{fig:fibers}, and study how the topology changes. All of the subfigures display level sets in $\mathbb{C}\times[0,2\pi]$, whose common boundary is the braid $B$. Identifying the bottom and the top plane, results in a solid torus $\mathbb{C}\times S^1$. Its complimentary solid torus in $S^3$ may be filled with meridional disks, so that a visualisation of a (pseudo-)fibration in $\mathbb{C}\times[0,2\pi]$ can be used to visualise the (pseudo-)fibration in $S^3$.

Some topology changes between surfaces are easier to spot than others. We see for example that between the fourth and fifth subfigure the genus of the surface increases by one. It then decreases between the sixth and the seventh subfigure, increases again between the eighth and ninth. Between the twelfth and thirteenth subfigure the genus finally decreases.

This superficial visual analysis makes it easy to miss critical points that occur in rapid succession, that is, pairs of critical points $p_1,p_2\in\mathbb{C}\times [0,2\pi]$ for which $\arg(g(p_1))$ and $\arg(g(p_2))$ are very close. Going by Figure~\ref{fig:fibers} we might conclude that $\arg(g)$ has four critical points. However, this is not correct.

Figure~\ref{fig:critvalues52} shows the graphs of $\arg(v_j(t))$, $j=1,2,$, where $v_j(t)$, $j=1,2$, are the critical values of $g_t$. Since critical points of $\arg(g)$ correspond to points with $\tfrac{\partial \arg(v_j(t))}{\partial t}=0$, we see immediately from this plot that $\arg(g)$ has exactly 6 critical points. Comparing this with $\MN(5_2)$, we see that it has more critical points than necessary.

\begin{figure}[H]
\includegraphics[height=4cm]{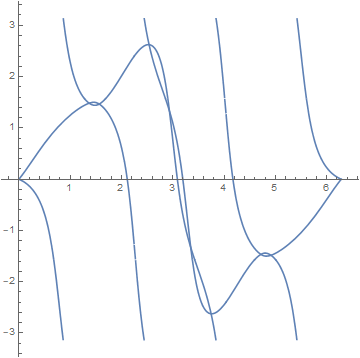}
\caption{The graphs of $\arg(v_j(t))$, $j=1,2$, where $v_j(t)$, $j=1,2$, are the critical values of $g_t$. \label{fig:critvalues52}}
\end{figure}

As discussed earlier, the critical points of $\arg(g)$ must lie on the saddle point braid, which is plotted in Figure~\ref{fig:saddlebraid}a). In Figure~\ref{fig:saddlebraid} we have coloured the saddle point braid by $\arg(g)$. In this way the critical points are seen as those points where the there is a change in the direction with which the colour wheel is traversed. Figure~\ref{fig:saddlebraid}b) also shows a level set of $\arg(g)$. Note that its intersection points with the saddle point braid are its points with horizontal tangent planes.

\begin{figure}
\centering
\includegraphics[height=7cm]{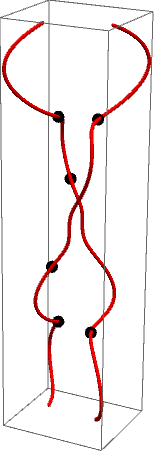}
\includegraphics[height=7cm]{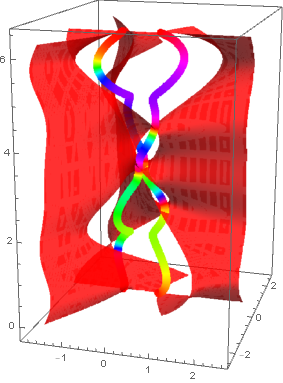}
\caption{a) The saddle point braid and the critical points of $\arg(g)$. b) The saddle point braid colored by $\arg(g)$ and a level set of $\arg(g)$.\label{fig:saddlebraid}}
\end{figure}

Knowing the location of the critical points allows us to show the topology changes in the level sets as a critical value is passed. Figure~\ref{fig:critical} shows such a change that is representative for all changes. Two sheets move towards each other as we approach the critical level set (Figure~\ref{fig:critical}a)). When we reach the level set of the critical value, the two parts meet in a double-cone, whose tip is a critical point of $\arg(g)$ (Figure~\ref{fig:critical}b)). Increasing the value of $\arg(g)$ further results in a level set with a compressing disk (Figure~\ref{fig:critical}c)). The genus has increased by 1. Of course, going through this sequence of pictures in the opposite order, starting with a compressing disk that shrinks to a double cone and splits into two separate sheets, decreases the genus.

\begin{figure}
\centering
\includegraphics[height=5cm]{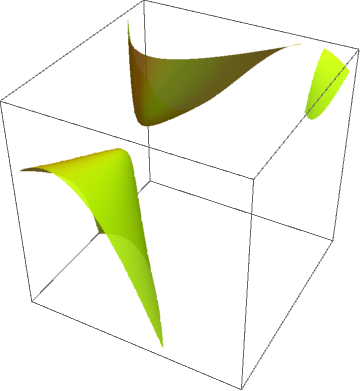}
\includegraphics[height=5cm]{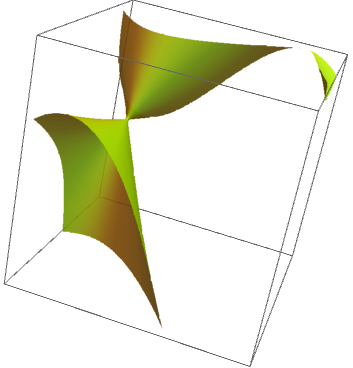}
\includegraphics[height=5cm]{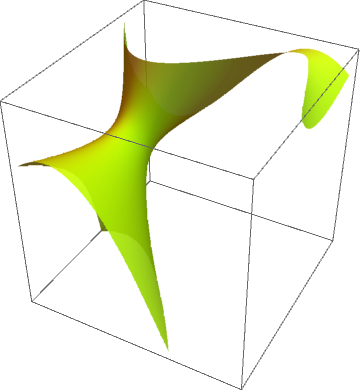}
\caption{A local change of topology of a level set of $\arg(g)$. Say $\chi\in S^1$ is the critical value of $\arg(g)$. a) Part of $\arg(g)^{-1}(\chi-\varepsilon)$ with $\varepsilon>0$ and small. b) Part of the level set $\arg(g)^{-1}(\chi)$. c) Part of the level set of $\arg(g)^{-1}(\chi-\varepsilon)$ with $\varepsilon>0$ and small.\label{fig:critical}}
\end{figure}

We know that the Morse-Novikov number of $5_2$ is 2, but $\arg(g)$ has 6 critical points. This immediately leads to the question: Could we have chosen a different parametrisation of a braid that closes to $5_2$ such that the resulting argument map of the polynomial only has two critical points? The answer to this question is ``Yes''. We will explain why and illustrate with this example a useful technique to prove that a given non-fibered link has Morse-Novikov number equal to 2.

Figure~\ref{fig:braidcritvalues} shows the braid that is formed by the critical values of $g_t$ and the strand $\{0\}\times [0,2\pi]$ both from the side and from the top. Note that the endpoints of the blue strand at the top and bottom match the endpoints of the green strand although the perspective in Figure~\ref{fig:braidcritvalues} makes that difficult to see.  The figures also show the six points, where $\tfrac{\partial \arg(v_j(t))}{\partial t}$ vanishes, corresponding to critical points of $\arg(g)$. 

\begin{figure}
\centering
\includegraphics[height=8cm]{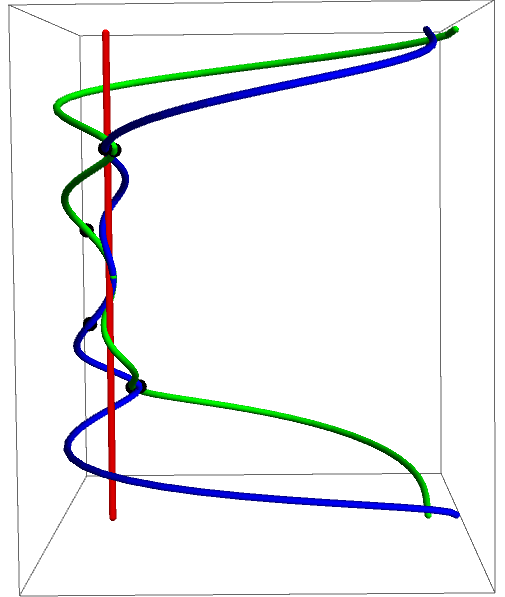}
\includegraphics[height=8cm]{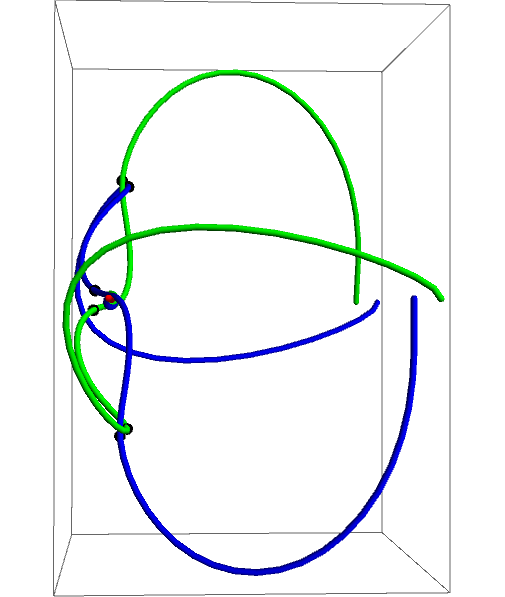}
\caption{The braid of critical values $v_j(t)$, $j=1,2$, in blue and green, together with the strand $\{0\}\times[0,2\pi]$ in red. Black points indicate the points where $\tfrac{\partial \arg(v_j(t))}{\partial t}$ vanishes, corresponding to critical points of $\arg(g)$. a) The view from the side. b) The view from the top. \label{fig:braidcritvalues}}
\end{figure}

We may deform the braid of critical values in $(\mathbb{C}\backslash\{0\})\times[0,2\pi]$ to the braid shown in Figure~\ref{fig:newcritbraid}, which only has two points where the critical values change their orientation. There are two pairs of such points in Figure~\ref{fig:braidcritvalues} that are very close together. In order to go from the braid in Figure~\ref{fig:braidcritvalues} to the one in Figure~\ref{fig:newcritbraid} we simply have to move the blue and the green strand  near this pairs of points. For the lower pair we pull the blue and the green strand from right to left behind the red strand $\{0\}\times [0,2\pi]$ until the critical point on the blue strand cancels with the next critical point on the blue strand. Likewise, for the upper pair of points we move the blue and the green strand from right to left in front of the red strand. In this case, two critical points on the green strand cancel.

\begin{figure}
\centering
\includegraphics[height=6cm]{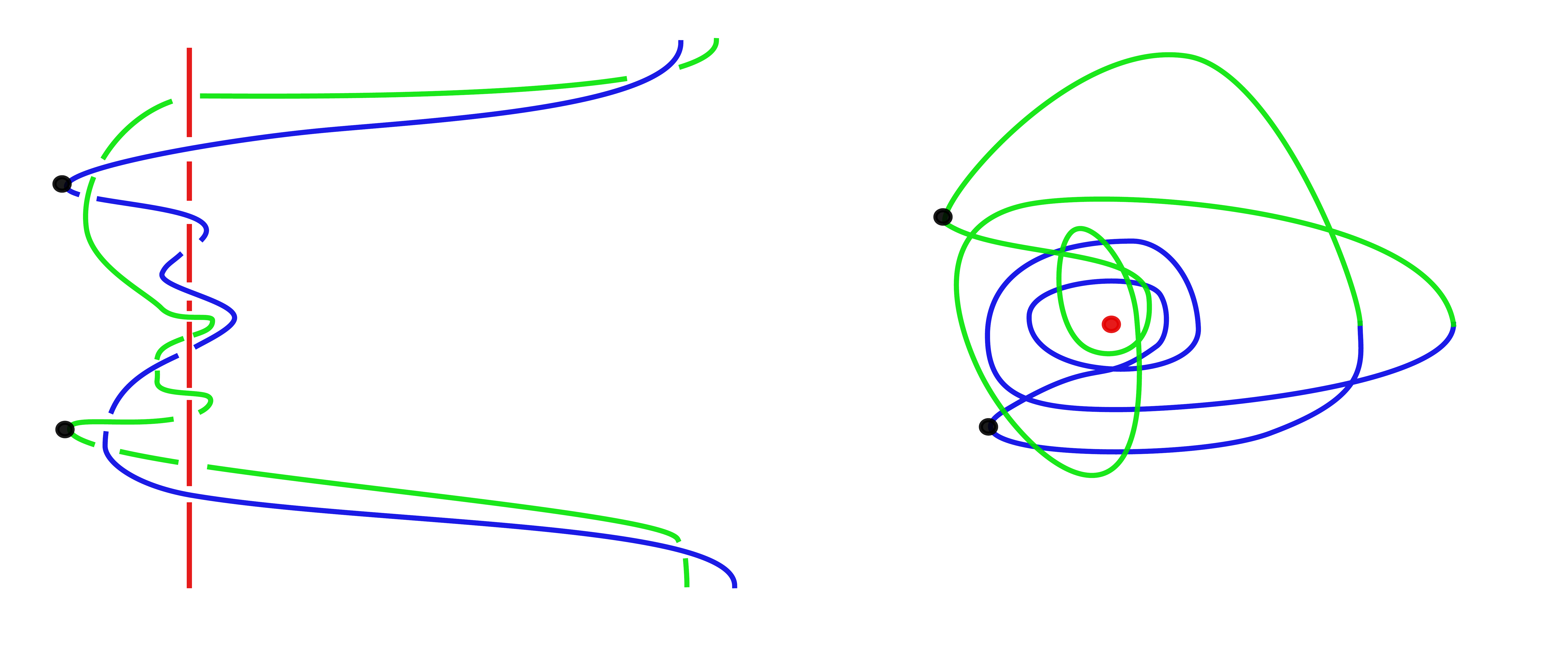}
\caption{The deformed braid of critical values. a) The view from the side. b) The view from the top.\label{fig:newcritbraid}}
\end{figure}

This deformation is a homotopy of the loop $\pi(\theta_n(g_t))$ in $V_n$, which extends to a homotopy of $\theta_n(g_t)$ in $\widehat{V}_n$. This homotopy lifts to a homotopy of $g_t$ in $\widehat{X}_n$, so that there is a loop of monic polynomials $h_t$ whose critical values form the braid in Figure~\ref{fig:newcritbraid}a), while its roots form the braid $B$ that closes to $5_2$. In particular, by construction $\arg(h)$, $h(u,\rme^{\rmi t}):=h_t(u)$, has exactly 2 critical points. Since $5_2$ is not fibered, its Morse-Novikov number must be at least 2. We have found a circle-valued Morse map with the required properties and exactly 2 critical points, which shows that $\MN(5_2)=2$. Of course, we already knew this before, but the method can be applied to any given non-fibered knot or link for which we can deform the braid of critical values in such a way that the number of argument-critical points reduces to 2. 

The algorithm in \cite{bode:polynomial} produces for any given link a loop of polynomials $g_t$. We may then study the corresponding loop of critical values $\pi(\theta_n(g_t))$ and try to deform it in a way that reduces the number of critical points. In general, this produces an upper bound on the Morse-Novikov number. However, if we achieve a deformation whose lift has an endpoint $h_t$ such that $\arg(h)$ does not have any critical points, then obviously the Morse-Novikov number is equal to zero. Likewise, if we already know that the link in question is not fibered (like $5_2$ above) and we achieve a deformation of the critical values such that the corresponding loop of polynomials $h_t$ results in exactly 2 critical points of $\arg(h)$, then the Morse-Novikov number must be equal to 2.

\subsection{Visualisations from Rampichini diagrams and films}\label{sec:rampi}

As pointed out, the method from the previous subsection requires the knowledge of an explicit function $g_t$. 
%Especially , we do not have this knowledge. 
%For many fibered links, for example the T-homogeneous braids from the previous section, we know that they are P-fibered, but we do not have an explicit expression for the corresponding loop of polynomials $g_t$. 
The problem with this approach is that if we want $g_t$ to have certain properties like inducing a fibration $\arg(g)$, usually (even for T-homogeneous braids) we do not know the loop $g_t$. We proved its existence in Proposition~\ref{prop:key}, but a key step in this proof was to deform a loop in $\widehat{V}_n$, which lifts to a homotopy of loops in $\widehat{X}_n$ whose endpoint is $g_t$. This lifting procedure requires solving a 1-parameter family of a set of polynomial equations, which is challenging even with the help of computers if the degree of the polynomials $g_t$ is $n>3$.

However, for any braid $B$ and any embedded tree $T$ we know the loop of critical values in $V_n$ and the basepoint $g_0$ of a loop $g_t$ such that $\arg(g)$ has exactly $\beta_T(B)$ critical points. This is sufficient to draw the Rampichini diagram of $g_t$ as in \cite{bode:braided}. This diagram encodes all the information we need in order to visualise the fibration.

A Rampichini diagram consists of a square that contains a number of curves that are labeled by transpositions. The square should be interpreted as a torus with both the horizontal and vertical edges corresponding to coordinate axis, going from 0 to $2\pi$. The vertical direction corresponds to $t$, the variable that parametrises the loop $g_t$. The horizontal direction corresponds to $\varphi=\arg(g)$. 

\begin{figure}[H]
\labellist
\small
\pinlabel 0 at 150 160
\pinlabel 0 at 210 100
\pinlabel $2\pi$ at 850 100
\pinlabel $2\pi$ at 150 800
\pinlabel $(2\ 3)$ at 270 260
\pinlabel $(1\ 4)$ at 270 460
\pinlabel $(3\ 4)$ at 270 630
\pinlabel $(1\ 4)$ at 410 840
\pinlabel $(2\ 4)$ at 530 840
\pinlabel $(3\ 4)$ at 690 840
\pinlabel $(2\ 4)$ at 415 660
\pinlabel $(2\ 3)$ at 440 465
\pinlabel $(1\ 4)$ at 440 320
\pinlabel  $(1\ 2)$ at 550 670
\pinlabel $(3\ 4)$ at 495 570
\pinlabel $(2\ 4)$ at 630 360
\pinlabel $(3\ 4)$ at 720 230
\pinlabel $(1\ 4)$ at 900 590
\pinlabel $(1\ 2)$ at 900 440
\pinlabel $(3\ 4)$ at 900 260
\pinlabel $(1\ 2)$ at 610 575
\Large
\pinlabel $t$ at 100 500
\pinlabel $\varphi$ at 500 50
\endlabellist
\centering
\includegraphics[height=7cm]{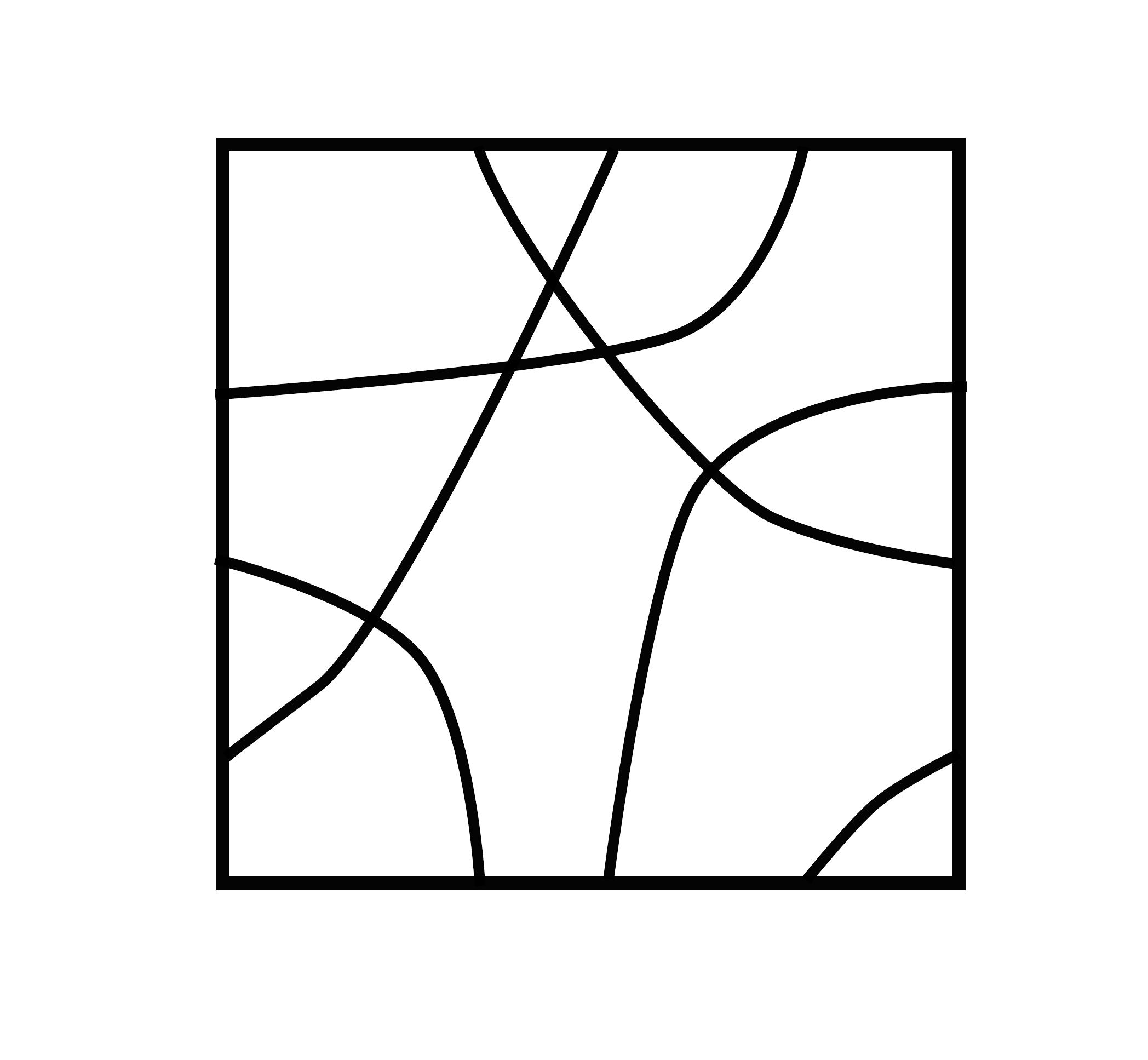}
\caption{A Rampichini diagram.\label{fig:rampichinidiagram}}
\end{figure}
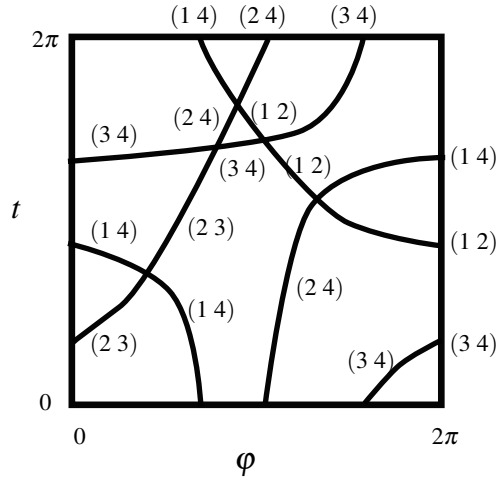

The curves in the square keep track of the motion of the critical values $v_j(t)$, $j=1,2,\ldots,n-1$, of $g_t$. More precisely, a point with coordinates $(\varphi_*,t_*)$ lies on a curve if and only if there is a $j\in\{1,2,\ldots,n-1\}$ such that $\arg(v_j(t_*))=\varphi_*$, where the critical values are ordered as usual $0<\arg(v_1(t))<\arg(v_2(t))<\ldots<\arg(v_{n-1}(t))<2\pi$.

Since the roots of $g_t$ are assumed to form a P-fibered geometric braid, all of the curves in the square are either strictly monotone increasing or strictly monotone decreasing, depending on the sign of $\tfrac{\partial\arg(v_j(t))}{\partial t}$ of the corresponding critical value $v_j(t)$. We may extend the definition of a Rampichini diagram to loops of polynomials $g_t$ whose roots do not form P-fibered geometric braids, in which case there are points where the curves in the square have vertical tangencies. Assuming a generic non-degeneracy condition, the number of such vertical tangencies is exactly the number of critical points of $\arg(g)$, $g(u,\rme^{\rmi t})=g_t(u)$. The term ``Rampichini diagram'' is reserved for loops $g_t$ that correspond to P-fibered braids. If $g_t$ does not correspond to a P-fibered braid, the resulting diagram is simply called a square diagram.

Since every point on a curve in the square corresponds to a critical value $v_j(t)$ it comes with a transposition $\tau_j$ defined from the singular foliation induced by $\arg(g_t)$. These transpositions label the points on the curves in the square. The labels (and thus also the cactus) only change at the right edge of the square and at intersection points of the curves, corresponding to a critical value $v(t)$ crossing the line $\arg(v(t))=0$ or two critical values having the same argument, respectively. Therefore, it is sufficient to label the arcs of the curves between such points. This means that for any fixed value of $t$ the labels in the square diagram at that fixed height $t$, read from left to right, are exactly the cactus of $g_t$. An example of a Rampichini diagram is displayed in Figure~\ref{fig:rampichinidiagram}.

%The details on Rampichini diagrams can be found in \cite{mortramp, rampichini, bode:braided, bode:thomo}. For us the important 

For each $t\in[0,2\pi]$ for which the critical values have distinct arguments the map $\arg(g_t)$ induces a singular foliation on $\mathbb{C}$, similar to \cite{obf}. As mentioned in the previous section, the roots of $g_t$ are elliptic singularities and the critical points are hyperbolic points. The Rampichini diagram essentially stores for every value of $t\in[0,2\pi]$ the information about the cactus of $g_t$, which is enough to draw the entire singular foliation of $\mathbb{C}$. This combinatorial structure only changes at finitely many values of $t\in[0,2\pi]$, so that we have a topological description or a visualisation of the fibration map $\arg(g)$. All of this is described in more detail in \cite{mortramp, rampi, bode:braided}. This approach to visualise fibrations is quite similar to a very recent visualisation of fibrations for homogeneous braids that has been offered by \cite{maggie}. 

Figure~\ref{fig:movie} shows a sequence of singular foliations, while the four roots trace out the braid $\sigma_3$. Rampichini calls such a sequence a film \cite{rampi}.

\begin{figure}[H]
\centering
\includegraphics[height=20cm]{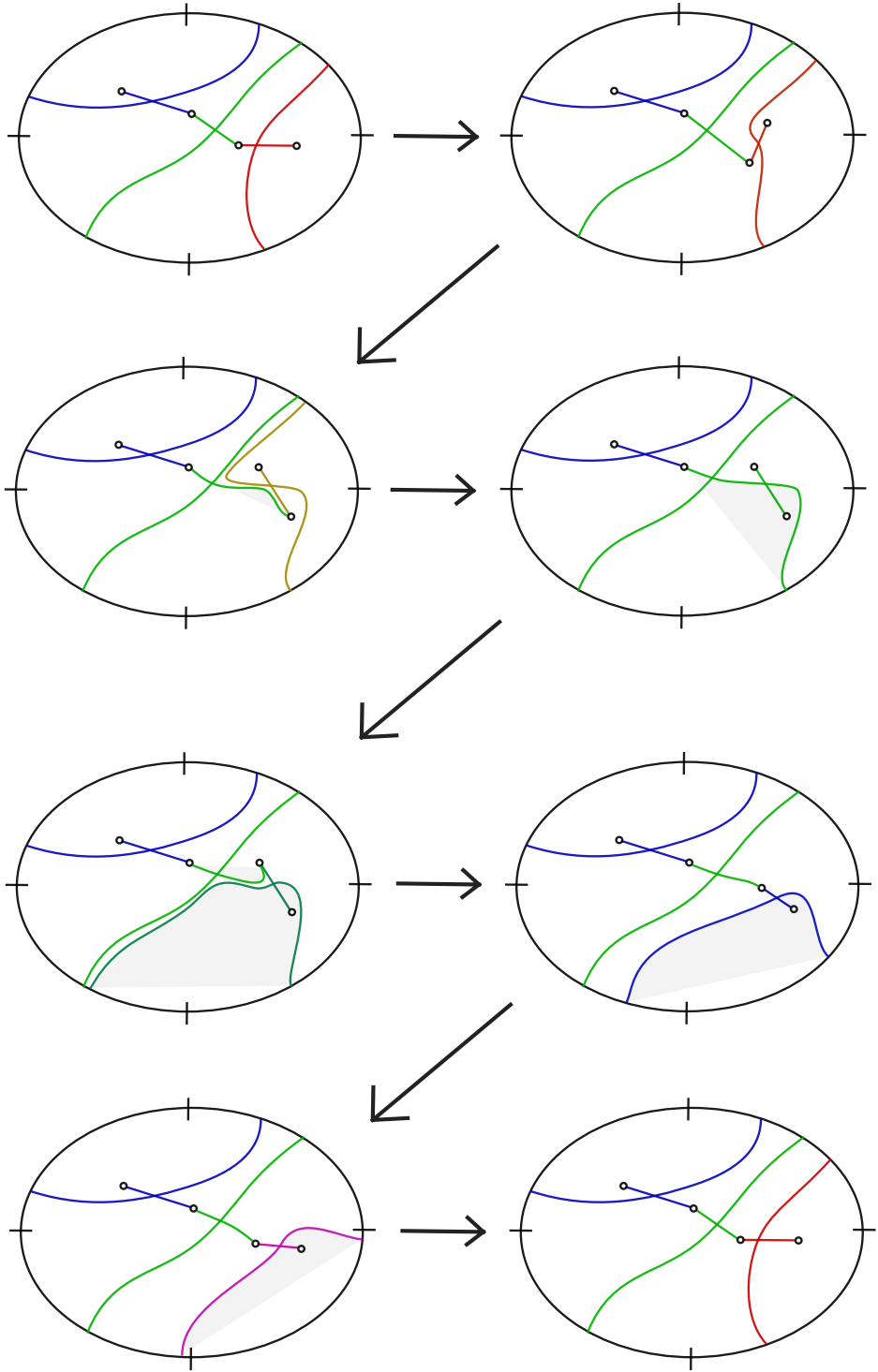}
\caption{A film of how the singular foliation on the disk changes during a crossing of the form $\sigma_3$.\label{fig:movie}}
\end{figure}

We can see in these figures that the saddle points form a trivial braid on 3 strands. Since the first picture is the same as the last, this sequence of figure represents a loop of polynomials. (The fact that such a loop is realised by polynomials is a consequence of Riemann's existence theorem \cite{bode:braided}.) We can thus compose loops of this form for any Artin generator and its inverse. In this way, Figure~\ref{fig:movie} represents a visual proof of Theorem~\ref{thm:saddle}.

Figure~\ref{fig:rampisigma1} shows part of a Rampichini diagram that corresponds to the film in Figure~\ref{fig:movie}. The dotted horizontal lines indicate the heights, that is, the values of $t$, for which Figure~\ref{fig:movie} shows the induced singular foliation. As in \cite{bode:thomo} we only include the transposition labels along an edge of the square and turn the intersections of lines in the square into crossings as in a (virtual) knot diagram with the convention that a line is the undecrossing arc if the corresponding critical value has smaller absolute value at the intersection. With this information and the rules from \cite{bode:braided, bode:thomo} we could reproduce the entire Rampichini diagram with all of its labels. 

The cactus of the polynomials corresponding to the first three pictures in Figure~\ref{fig:movie} is $\tau_1=(2\ 3)$, $\tau_2=(2\ 4)$, $\tau_3=(1\ 2)$. The fourth figure corresponds to a value of $t$ for which there is a crossing in Figure~\ref{fig:rampisigma1}. Therefore, we have $\arg(v_1)=\arg(v_2)$ and there is no well-defined cactus. The cactus for the fifth figure is $\tau_1=(2\ 4)$, $\tau_2=(3\ 4)$, $\tau_3=(1\ 2)$. The sixth figure has again has no well-defined cactus, since two critical values have the same argument. The cactus of the seventh figure is $\tau_1=(2\ 4)$, $\tau_2=(1\ 2)$, $\tau_3=(3\ 4)$. The last cactus is equal to the first one.

Starting from any Rampichini diagram we can draw a film as in Figure~\ref{fig:movie} that completely describes the fibration.

Rampichini diagrams offer another way to visualise fibrations. Instead of using transpositions to label the curves in the square, we may use band generators, so that $(i\ j)$ is replaced by $a_{i,j}$ if the corresponding line is strictly monotone increasing (or, equivalently, $\tfrac{\partial \arg(v_j(t))}{\partial t}>0$ for the corresponding critical value $v_j(t)$) and $(i\ j)$ is replaced by $a_{i,j}^{-1}$ if the line is strictly monotone decreasing (or, equivalently, $\tfrac{\partial \arg(v_j(t))}{\partial t}<0$ for the corresponding critical value $v_j(t)$).

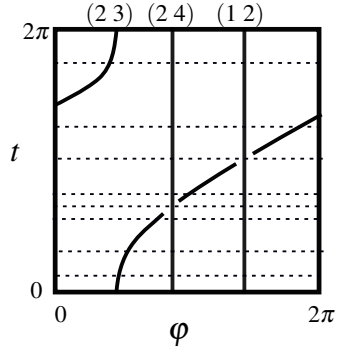
\begin{figure}
\labellist
\small
\pinlabel 0 at 150 160
\pinlabel 0 at 210 100
\pinlabel $2\pi$ at 850 100
\pinlabel $2\pi$ at 150 800
\pinlabel $(2\ 3)$ at 330 840
\pinlabel $(2\ 4)$ at 480 840
\pinlabel $(1\ 2)$ at 650 840
\Large
\pinlabel $t$ at 100 500
\pinlabel $\varphi$ at 500 50
\endlabellist
\centering
\includegraphics[height=5cm]{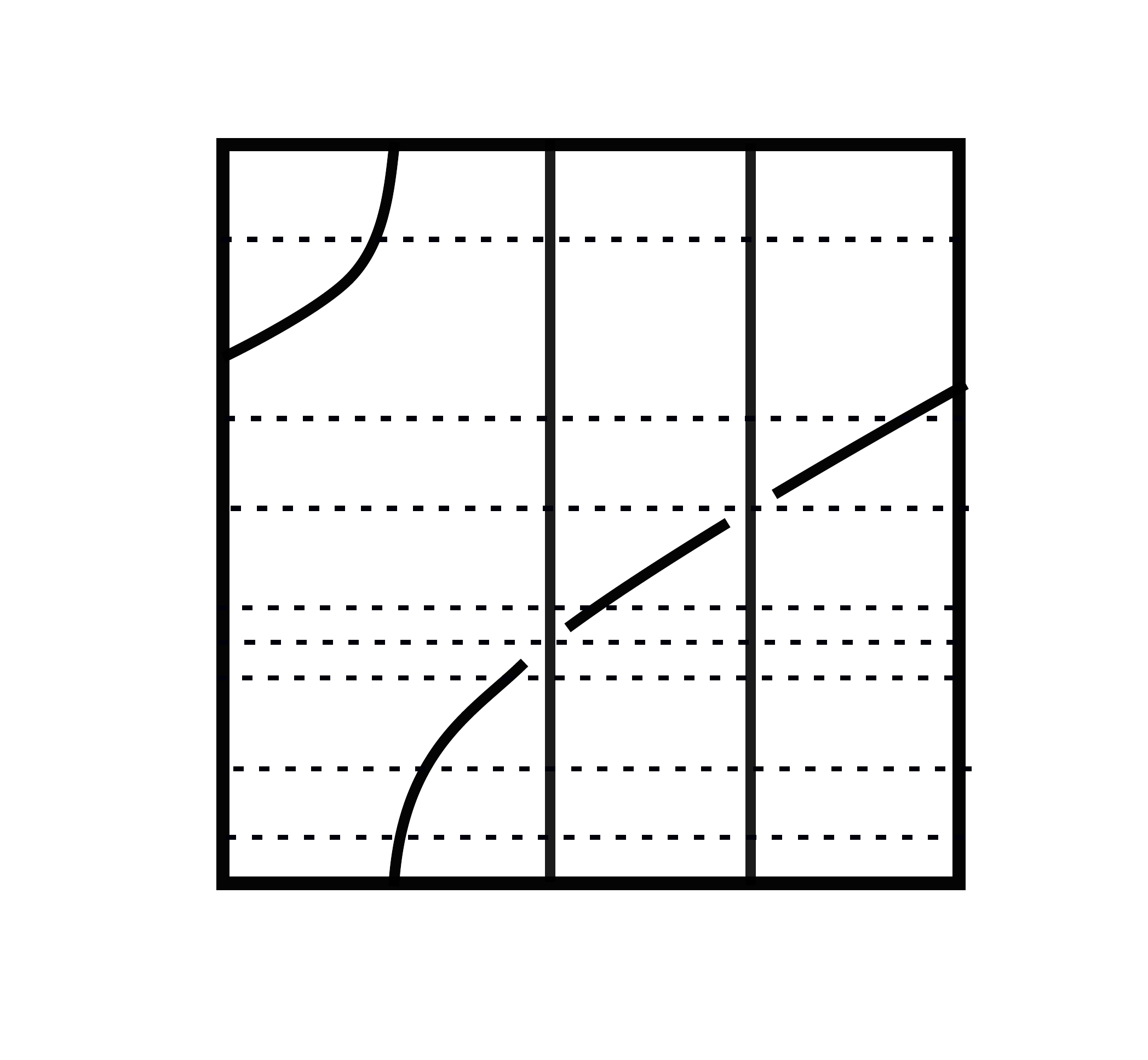}
\caption{Rampichini diagram corresponding to the film in Figure~\ref{fig:movie}.\label{fig:rampisigma1}}
\end{figure}

We obtain a word in BKL-generators for the fiber surfaces of the fibration as follows. We fix a value $\arg(g)=\varphi_*$, which corresponds to a vertical line in the Rampichini diagram. Reading the BKL-labels of the curves in the Rampichini diagram at intersection points with this vertical line, going from the bottom to the top, spells a BKL-word for the fiber $\arg(g)^{-1}(\varphi_*)$. By varying $\varphi_*$ we obtain a finite sequence of band words, representing the (topologically equivalent) braided surfaces.

An important aspect of this visualisation and the one obtained from a film is the insight that the saddle point braid consists exactly of those points where a fiber surface intersects the corresponding horizontal plane tangentially. That is, $c_j(t)$ is a critical point of $g_t$ if and only if the intersection of the level set $\arg(g)=\arg(v_j(t))=\arg(g_t(c_j(t)))$ and $\mathbb{C}\times \{t\}$ is tangential.

All of the fiber surfaces are braided surfaces and therefore consist of $n$ disks that are connected by a number of half-twisted bands. Each half-twisted band has exactly one such saddle point.

This is a reason why the topological type of the saddle point braid on its own is not that helpful for the visualisation. It tells us where the saddle points are, but the crucial piece of information, which of the disks are connected by the corresponding band, is missing. For this we need the additional information of the tree or the cactus of the polynomial, which is stored in the Rampichini diagram.

\subsection{Fibrations for homogeneous braids}

The visualisation technique from Section~\ref{sec:rampi} offers a complete graphical description of the fibration without the knowledge of an explicit expression for the fibration map or the loop of polynomials $g_t$. Furthermore, it illustrates that for T-homogeneous braids the saddle point braid may be chosen to be the trivial braid. However, it requires some extra work to go from a sequence of figures as in Figure~\ref{fig:movie} to a sequence of figures as in Figure~\ref{fig:fibers}, which display how the different fiber surfaces sweep out the link complement. In this section, we will describe a new visualisation technique for the fibrations of homogeneous braids. As in Section~\ref{sec:rampi} we do not need to know an expression of the corresponding polynomials $g_t$ and the saddle point braid is easily seen to be the trivial braid. Yet we arrive at a graphical description of the fibration in terms of a sequence of fiber surfaces, which makes it perhaps more practical than the technique from Section~\ref{sec:rampi}.

The rough idea of this technique is displayed in Figure~\ref{fig:mikami}a). The starting point is a banded surface, whose boundary is the closure of a homogeneous braid on four strands. The fiber surface is thus realised by four disks connected by a number of bands, one for each crossing of the braid. Note that the sign of the crossings within any row does not change, so that the braid is indeed homogeneous. We now push the top disk up until it is located as in the second image in Figure~\ref{fig:mikami}a). The disk is now pointing outwards, that is, having started with a non-zero section of tangent vectors $v$ of the disk, normal to its boundary, pushing the disk as described and moving the tangent vectors with it results exactly in $-v$. Note that now the opposite side of the surface is pointing upwards. In Figure~\ref{fig:mikami}a) the two sides of the surface are coloured blue and red, respectively.

In order to move from the second image to the third image, we have to push the red disk on the outside down, while moving blue disk in the layer below it upwards. This is tricky. Later we will describe in more detail how this is done and what happens with the bands between these two disks. For now, the reader will have to accept that there is such a motion of surfaces fixing the boundary that moves the red disk on the outside one layer downwards, while moving the blue disk on the inside one layer up, resulting in the third image. This motion is repeated until we reach the fifth and last image in Figure~\ref{fig:mikami}a), where the red disk on the outside has arrived at the bottom layer.

We may now pull the red disk on the outside down, while keeping its boundary fixed. It passes through the point at infinity (if we think of our images as placed in $\mathbb{R}^3$) and eventually sits on the inside with its blue side facing up. We have thus reached the first image again. Note that this sequence of surfaces fills the entire link complement, they can be made disjoint and they are isotopies with fixed boundary of the original Seifert surface. Thus, if we have a convincing picture for the motion that lowers a red outside disk and raises a blue inside disk one layer, we have a visualisation of the fibrations for homogeneous braids. Note that this approach is philosophically similar to Stallings' original proof in \cite{stallings} that closures of homogeneous braids are fibered. It is built on the fact that surfaces that are obtained from fiber surfaces via an operation called Murasugi sum are themselves fiber surfaces of a fibration. The surfaces for a homogeneous braid are obtained by repeated Murasugi sums of fiber surfaces for braids on two strands. Stallings' proof therefore reduces to the (known) case of braids on two strands. Similarly, our visualisation technique is reduced to the particular case of a braid on two strands and the general case of a homogeneous braid is obtained by stacking the pictures as in Figure~\ref{fig:mikami}a).

Figure~\ref{fig:mikami}b) and c) illustrate the motion of the surface that lowers the red outside disk and raises the blue inside disk. It is very important to understand that the displayed surface is not meant to be a fiber surface or part of a fiber surface. Instead, it visualises the motion of a part of the fiber surface near a band. Starting at the very left of either figure we see a red sheet on top, pointing to the front (outside), and a blue sheet at the bottom pointing into the diagram plane (inside). The two sheets are connected by a half-twisted band as usual. The shading in Figure~\ref{fig:mikami}c) indicates the two sides of the surface, blue and red. Figure~\ref{fig:mikami}b) highlights the contours of the surface, that is, the fold lines where the there is change of which side of the surface is facing the reader and thus a change of colour. Note that all surfaces are smooth. The contours are an attribute of the chosen perspective on the surface.

Going from left to right in Figures~\ref{fig:mikami}b) and \ref{fig:mikami}c) we see how the surface is deformed near a band until at the very right we see a blue sheet at the top, pointing to the back (inside), and a red sheet at the bottom pointing to the front (outside), again connected by a half-twisted band. Therefore, the left end of the figures shows a part of the surface at the beginning of the motion and the right end of the figure shows the desired endpoint of the motion. We now describe the intermediate steps in more detail.

\begin{figure}
\labellist
\Large
\pinlabel a) at 20 550
\pinlabel b) at 20 370
\pinlabel c) at 20 200
\endlabellist
\centering
\includegraphics[height=13cm]{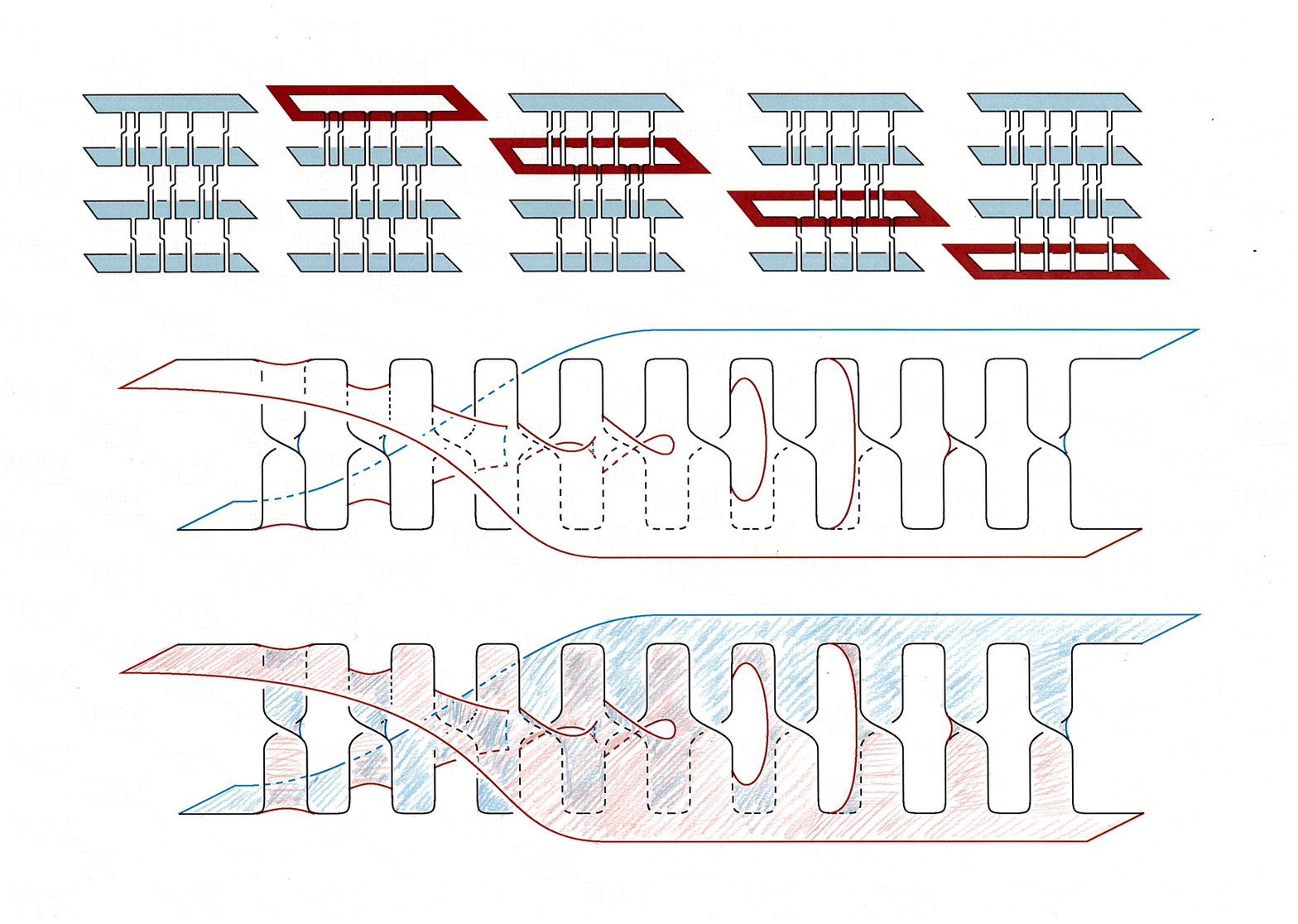}
\caption{Fibrations for homogeneous braids. a) A visualisation of several fiber surface of a fibration for a homogeneous braid on four strands. b) An isotopy of a fiber surface that lowers the red outside disk and raises the blue inner disk by one level, highlighting the motion of the contour lines. Going from left to right, we see how a fiber surface is isotoped near a band. c) The same isotopy, now with a coloured shading that indicates the two sides of the surface. \label{fig:mikami}}
\end{figure}

In order to do this, it is helpful to remember some basics on projections of surfaces. Figure~\ref{fig:contour}a) shows a part of a surface that is folded in way that creates two contour lines, one blue and one red. The two lines meet at a point as illustrated. Again we want to emphasize that all surfaces are smooth. The intersection point between the two contours is not a singular point of the surface. It depends on the chosen perspective on the surface and might be understood as a feature of the corresponding projection map from the surface to the diagram plane, rather than an artefact of the surface itself. This shows how we can apply an isotopy to a flat piece of a surface, i.e., without cusps or fold lines, to obtain a surface with two fold lines meeting in a point. Naturally, applying the inverse isotopy brings us from Figure~\ref{fig:contour}a) to a flat, smooth piece of surface.

Figure~\ref{fig:contour}b) displays the neighbourhood of a saddle point of a surface with one red contour line. We may place a finger on the saddle point and push the saddle point and the surface downwards. If our finger is not orthogonal to the diagram plane, the resulting surface displays two cusps as in Figure~\ref{fig:contour}c). The single red contour line has split into two red lines, each connecting the boundary to a cusp, and a blue contour that is connecting the two cusps. We may think of this as a first Reidemeister move applied to the fold line. Again, the cusps are not singular points of the surface. We thus have an isotopy from Figure~\ref{fig:contour}b) to Figure~\ref{fig:contour}c) and its inverse allows us to cancel two cusps as in Figure~\ref{fig:contour}c).

\begin{figure}
\labellist
\Large
\pinlabel a) at 20 140
\pinlabel b) at 143 140
\pinlabel c) at 285 140
\endlabellist
\centering
\includegraphics[height=5cm]{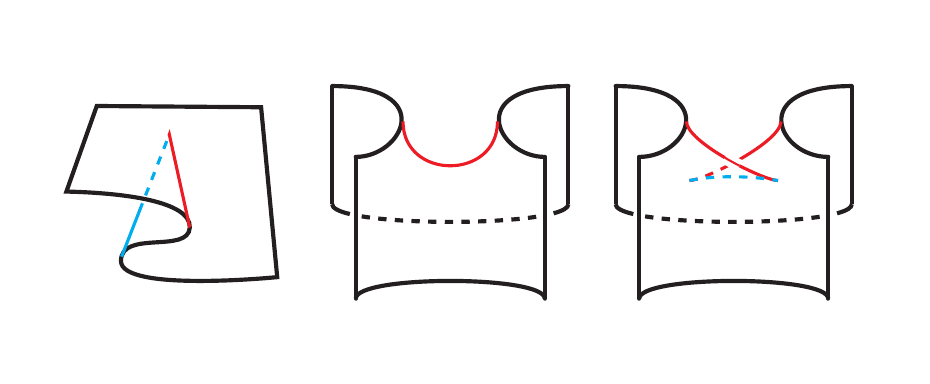}
\caption{Folds and cusps. a) A folded, smooth surface with two contour lines meeting in a point. b) A saddle point of a surface with one contour line. c) After an isotopy the surface displays two cusps and three contour lines. \label{fig:contour}}
\end{figure}

\begin{wrapfigure}{r}{0.4\textwidth}
\centering
\includegraphics[width=0.12\textwidth]{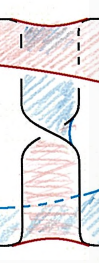}\qquad\quad
\includegraphics[width=0.12\textwidth]{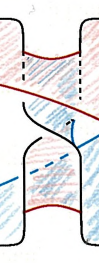}
\caption{The initial surface near a crossing and its change after the first step of the isotopy.\label{fig:mikami2}}
\end{wrapfigure}
%\begin{wrapfigure}{r}{0.2\textwidth}
%\centering
%\includegraphics[width=0.15\textwidth]{Mikami_figure2}
%\caption{Pulling the horizontal contours towards the crossing and the vertical contour away from it.}
%\end{wrapfigure}
We now return to Figures~\ref{fig:mikami}b) and c) to discuss the motion of the surface. As mentioned before we start at the left end of the figure, where we have two sheets connected by a half-twisted band. It is also displayed on the left of Figure~\ref{fig:mikami2}. The top red surface extends out of the diagram plane towards the reader, while the bottom blue surface extends into the diagram plane. There are three fold lines (or contours) and no cusps. Each contour connects two points on the boundary braid. We have two contours at the top and the bottom of the band and one almost vertical contour next to the crossing. Which side of the crossing it is on depends on the sign of the crossing.

Moving our gaze to the right to the next band in Figure~\ref{fig:mikami}c), we see the result of a small isotopy to our original surface in a neighbourhood of the same band. It is also shown on the right of Figure~\ref{fig:mikami2}. Both at the bottom and the top of the band we may hook a finger around the contour and pull both of them towards the crossing. At the same time the piece of the surface next to the crossing and with the nearly vertical contour is pushed to the right, away from the crossing. Pulling the horizontal contours towards the crossing means that their endpoints on the boundary braid also move towards the crossing. Similarly, pushing the vertical contour away from the crossing makes its two endpoints on the boundary move away from the crossing. Eventually, the endpoints of the horizontal contours have to meet the endpoints of the vertical contour. Pushing these two intersection points of the contours into the inside of the surface creates two cusps, corresponding to the points where the red (horizontal) and blue (vertical) contours meet, as illustrated with the next band in Figures~\ref{fig:mikami}b) and c), also shown on the left of Figure~\ref{fig:mikami3}.

\begin{wrapfigure}{r}{0.6\textwidth}
\centering
\includegraphics[width=0.15\textwidth]{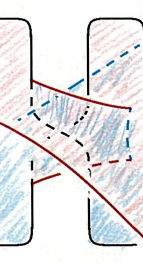}\quad\quad
\includegraphics[width=0.13\textwidth]{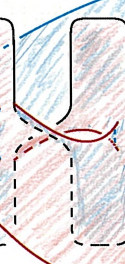}\quad\quad
\includegraphics[width=0.20\textwidth]{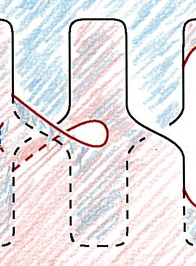}
\caption{The surface after the second, third and fourth step of the isotopy near a band.\label{fig:mikami3}}
\end{wrapfigure}
If we keep pulling the upper horizontal contour down and the lower horizontal contour up, they pass each other as in a second Reidemeister move, see the middle part of Figure~\ref{fig:mikami3} or the fourth band in Figures~\ref{fig:mikami}b) and c). We may now apply the isotopy going from Figure~\ref{fig:contour}c) to Figure~\ref{fig:contour}b), which corresponds to a first Reidemeister move and cancels the two cusps. Note that all of these isotopies of the corresponding surfaces fix the boundary braid. The analogy with Reidemeister moves refers to the motion of the contours, not the boundary braid. The resulting surface is shown on the right in Figure~\ref{fig:mikami3}, which is the fifth band in Figure~\ref{fig:mikami}b) and c). The resulting contour starts a bit above the crossing, goes downwards in front of the crossing, forms a small loop and passes behind the original band back to the boundary braid.

\begin{wrapfigure}{r}{0.6\textwidth}
\centering
\includegraphics[width=0.18\textwidth]{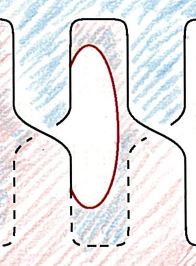}\quad\quad
\includegraphics[width=0.17\textwidth]{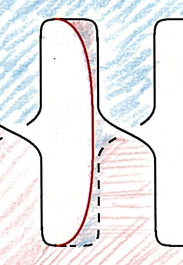}\quad\quad
\includegraphics[width=0.14\textwidth]{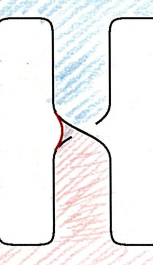}\quad\quad
\caption{The surface after the fifth, sixth and seventh step of the isotopy near a band.\label{fig:mikami4}}
\end{wrapfigure}
In the next step we grow this loop until it fills the entire gap between the original band and the next band to the right. In doing so, the endpoints of the contour move past the crossing until they lie on the horizontal parts of the boundary braid between two bands, see the left and middle part of Figure~\ref{fig:mikami4}. 

We may move the contour further towards the band that is directly to the right of the band that we started with until it is a small vertical contour directly to the left of the crossing as seen in the penultimate band in Figures~\ref{fig:mikami}b) and c) as well as on the right of Figure~\ref{fig:mikami4}. Pushing the surface near the crossing towards the right, produces a blue vertical contour directly to the right of the crossing as seen in the last band of Figures~\ref{fig:mikami}b) and c). We have described most of this motion of the surface in terms of the motion of the contour lines and the creation and cancellation of cusps. Figure~\ref{fig:mikami}c) shows how that motion extends to the surface. That is, throughout the motion the red outside disk is pushed downwards and the blue inside disk is pushed upwards, giving us the desired isotopy.

Note that this isotopy moves a band of the original surface to the band directly to the right of it. This is understood as a cyclic operation, that is, the last band is moved to the first band. The fact that all bands are moving to the right is owed to the fact that in the initial surface the blue contours in this example lie to the right of the crossings. If we had pictured crossings with the opposite sign, then the corresponding isotopy would move all bands to the left. This illustrates why this techniques requires that the braid in question is homogeneous. In every row the crossings need to have the same sign, so that there is in every row one direction in which all the bands are pushed.

For fiber surfaces in a braided open book the number of saddle points equals the number of bands, which is also the number of crossings if the corresponding braid is homogeneous. The surface in Figure~\ref{fig:mikami}c) is not a fiber surface, but a convenient visualisation of the isotopy between different fiber surfaces, combining the information from Figure~\ref{fig:mikami2} through Figure~\ref{fig:mikami4}. We can see ten crossings and nine saddle points. This is because the isotopy moves every band of a fiber to the crossing to its immediate right (or to the left if the crossings have the opposite sign). In the visualisation of the isotopy we thus have to display a band and the band to its immediate right, see for example the right part of Figure~\ref{fig:mikami3}, so that we end up with more crossings than saddle points.

Figures~\ref{fig:mikami}b) and c) also give a visual proof of Theorem~\ref{thm:thomosaddle} for homogeneous braids. We know that homogeneous braids are P-fibered. As mentioned in the previous section this implies that the corresponding open book in $S^3$ can be braided, that is, for a homogeneous braid $B$ there is a braid axis $O$ that is positively transverse to the fiber surfaces of $B$. The saddle point braid of the braided open book may then be defined without referring to a loop of polynomials $g_t$. It is given by the points of tangential intersections between the fiber surfaces of $B$ and the fiber disks of $O$. In Figure~\ref{fig:mikami}a) the fiber disks are vertical, so that the saddle points are the points where the tangent plane of the fiber surface is spanned by the vertical direction and the normal to the diagram plane. We can see from Figure~\ref{fig:mikami}c) that throughout the isotopy the saddle points are moving horizontally to the right. Again, if the sign of the crossing had been different, the saddle points would move to the left. It follows that the saddle points form a split unlink on $n-1$ components, where $n$ is the number of strands of the homogeneous braid $B$. Furthermore, we may choose an orientation for this unlink so that it becomes the trivial braid on $n-1$ strands relative to the same braid axis $O$.

A comment on this issue of orientation is in order. Since all the surfaces are oriented, the components of the derived bibraid, i.e., the link that is formed by the saddle points, split into two families, $pos$ and $neg$, see \cite{rudolph2}. Each saddle point $p$ is a tangential intersection point between a fiber surface and a fiber disk of $O$. Its component is in $pos$ if the orientations of the fiber surface and the fiber disk match at $p$. It lies in $neg$ if the two orientations do not match. Note that for homogeneous braids, this corresponds precisely to the sign of the corresponding band. Saddle points in rows with positive crossings lie on components in $pos$ and saddle points in rows with negative crossings lie in $neg$.

In earlier sections, when we discussed the saddle point braid in terms of loops of polynomials $g_t$, we naturally parametrised the components of the link consisting of saddle points by the variable $t$, which parametrises the loop $g_t$. Note that different values of $t\in S^1$ correspond to different fiber disks whose boundary is the braid axis $O$. In other words, we automatically picked an orientation for the derived bibraid that turned it into a braid relative to $O$, so that it is positively transverse to the fiber disks with boundary $O$.

In this subsection, we studied the motion of the saddle points as we vary the fiber surfaces (as opposed to varying the fiber disks). We can interpret this as parametrising the derived bibraid by the variable $\varphi\in S^1$ that parametrises the family of fiber surfaces. With the induced orientation the derived bibraid is positively transverse to the fiber surfaces, but in general it is not positively transverse to the fiber disks with boundary $O$. Therefore, with this induced orientation the derived bibraid is not a braid relative to $O$. Flipping the orientation of the components in $neg$ rectifies this and turns the derived bibraid into a braid with braid axis $O$, the saddle point braid.

%In order to obtain the real saddle point braid, we would have to flip the orientations of the components in $neg$.

%location of critical points, visualisation from critical values and Rampichini diagram (movie of singular foliations)

%Very recently, a similar visualisation of fibrations for homogeneous braids has also been offered by \cite{maggie}. 

\section{Singularities of semiholomorphic polynomials}\label{sec:sing}

Theorem~\ref{thm:semi} concerns the realisation of link types as the link of a weakly isolated singularity of a semiholomorphic polynomial $f$ while at the same time prescribing the link of the singularity of $f_u$. This type of topological flexibility of the jet of $f$ (while maintaining the link type $L$) may appear like a typical feature of real algebraic geometry. However, we will see in an example that this flexibility concerning the link of $f_u$ already appears in the complex setting, which is otherwise known to be very rigid.

First we recall some basic definitions concerning the Newton boundary of a mixed polynomial map $f:\mathbb{C}^2\to\mathbb{C}$ as described in \cite{oka}. We may write $f$ as $f(u,v)=\sum_{i,j,k,\ell\geq 0}c_{i,j,k,\ell}u^i\bar{u}^jv^k\bar{v}^\ell$ with all but finitely many coefficients $c_{i,j,k,\ell}$ equal to zero. The support of $f$, denoted by $supp(f)$ consists of all integer lattice points $(\mu,\nu)\in\mathbb{Z}^2$ such that there are non-negative integers $\mu_1,\mu_2,\nu_1,\nu_2$ with $\mu=\mu_1+\mu_2$, $\nu=\nu_1+\nu_2$ and $c_{\mu_1,\mu_2,\nu_1,\nu_2}\neq 0$.

The Newton polygon of $f$ is the convex hull of $supp(f)+(\mathbb{R}^+)^2$. Its boundary consists of a number of vertices and edges. The union of the vertices and compact edges is called the Newton boundary $\Gamma_f$ of $f$. There are several properties of $\Gamma_f$ that can be used to determine if $f$ has a (weakly) isolated singularity, such as convenience and Newton non-degeneracy \cite{oka} as well as inner non-degeneracy \cite{eder} or partial non-degeneracy \cite{eder2}. In the case of convenience and Newton non-degeneracy, and inner non-degeneracy it is known that the link of the singularity only depends on the terms of $f$ whose corresponding integer lattice points lie on $\Gamma_f$. The sum of these terms is called the principal part of $f$. In other words, adding terms above the Newton boundary does not change the fact that we have a (weakly) isolated singularity and it does not affect the link type. For more details on the Newton boundary of mixed polynomials we point the reader to \cite{oka, eder, eder2}.

\begin{example}
Consider the well-known example of $f(u,v)=u^p-v^q$ with $p,q\in\mathbb{N}$ with $p,q>1$, with an isolated singularity at the origin and the $(p,q)$-torus link $T_{p,q}$ as the link of the singularity. Since $f$ is holomorphic, we may also study $f_u$ and $f_v$ with regards to their singularities. However, it is easily seen that neither $f_u$ nor $f_v$ has a weakly isolated singularity at the origin.

Instead we may consider $F(u,v)=f(u,v)-uv^k+vu^{\ell}$ with $k,\ell\in\mathbb{N}$ with $k\geq q$ and $\ell\geq p$. First of all, since both added terms lie above the Newton boundary of $f$ and $f$ is convenient and Newton non-degenerate, it follows that $F$ has an isolated singularity at the origin, whose link is again the $(p,q)$-torus link $T_{p,q}$. But now we have $F_u(u,v)=pu^{p-1}+v^k$, which has an isolated singularity with link $T_{(p-1,k)}$, and $F_v(u,v)=u^{\ell}-qv^{q-1}$, which has an isolated singularity with link $T_{(\ell,q-1)}$. We thus have a lot of freedom for the links of $F_u$ and $F_v$ without changing the link of $F$.
\end{example}

This illustrates that the links of singularities of $F_u$ are not topological invariants. There are different equivalence relations on polynomial map germs with (weakly) isolated singularities related to topological properties (R-equivalence, A-equivalence, V-equivalence...). For example, we might say that two polynomials $f_1$ and $f_2$ are equivalent if there is a homeomorphism (or diffeomorphism) $h_1:(B_\varepsilon^4,0)\to (B_\varepsilon^4,0)$ and $h_2:(B_\varepsilon^2,0)\to (B_\varepsilon^2,0)$ of the 4-ball and 2-ball of radius $\varepsilon$, respectively, such that on $B_\varepsilon^4$ we have $f_2=h_2\circ f_1\circ h_1$, so that in particular the link of $f_1$ is ambient isotopic to that of $f_2$. While the link types of the singularities are topological invariants, that is, they do not depend on the representative of the equivalence class of polynomial maps, we cannot expect the same to be true for the link types of $f_u$. That is, in general we should expect that the link of $(f_1)_u$ is different from that of $(f_2)_u$ even if $f_1$ and $f_2$ are in the same equivalence class. After all, even the property of being semiholomorphic itself depends on the particular variables and therefore changes depending on which representative of the equivalence is under consideration. The links of singularities of first derivatives thus may be used to define much finer equivalence relations that treat the polynomial maps as jets. For this note that there is no particular reason to restrict to semiholomorphic polynomial and derivatives with respect to the complex variable. If $x_1,x_2,x_3,x_4$ are the real coordinates on $\mathbb{R}^4$, we may say that two real polynomial maps $f_1,f_2:\mathbb{R}^4\to\mathbb{R}^2$ are link-equivalent as 1-jets if both have (weakly) isolated singularities with ambient isotopic links and $(f_i)_j$ has a (weakly) isolated singularity for all $i=1,2$, $j=1,2,3,4$, with the link of $(f_1)_j$ ambient isotopic to the link of $(f_2)_j$. We might consider such an equivalence relation up to permutation of links of $(f_1)_j$ so that a simple permutation of the variables does not change the equivalence class. If a polynomial $f$ is semiholomorphic with $u=x_1+\rmi x_2$, then $V_{f_u}=V_{f_{x_1}}=V_{f_{x_2}}$, so that the link types of the real derivatives contain the information about the link of the derivative with respect to $u$.

We now turn our attention to the proofs of Theorem~\ref{thm:semi} and Theorem~\ref{thm:thomosemi}. First, we review some techniques to turn loops $g_t$ of polynomials as in the previous sections into semiholomorphic polynomials with (weakly) isolated singularities.

Since trigonometric polynomials are $C^1$-dense in the space of $2\pi$-periodic, $C^1$-functions, we may approximate any loop of polynomials of fixed degree $n$ by a loop $g_t$ whose coefficients are polynomials in $\rme^{\rmi t}$ and $\rme^{-\rmi t}$. Then we may construct a function $f:\mathbb{C}^2\to\mathbb{C}$ with a singularity at the origin via
\begin{equation}\label{eq:ffromg}
f(u,r\rme^{\rmi t})=r^{kn}g_t\left(\frac{u}{r^k}\right),
\end{equation}
where $k$ is some sufficiently large natural number.  The resulting function $f$ is a holomorphic polynomial with respect to $u$, but is only a polynomial in the variables $v=r\rme^{\rmi t}$ and $\bar{v}=r\rme^{\rmi t}$ if $g_t$ (and thus its roots as well) satisfy certain symmetry requirements, namely, $g_{t+\pi}=g_t$ or $g_{t+\pi}(u)=-g_t(-u)$, as shown in \cite{bode:inner2}. If the roots of $g_t$ are distinct, then the singularity at the origin is weakly isolated and its link is the closure of the braid formed by the roots of $g_t$. If the roots of $g_t$ form a P-fibered geometric braid, then the singularity is isolated.

The first author showed in \cite{bode:ak} that every link type arises as the link of a weakly isolated singularity of a semiholomorphic polynomial. In \cite{bode:thomo}, it was proved that closures of $T$-homogeneous braids are real algebraic. Both proofs are constructive and are based on a variation of the idea described above.

In both cases, we start with a loop $g_t$ that satisfies the desired symmetry constraints. However, its roots are not distinct for all $t\in[0,2\pi]$, so that they do not form a geometric braid, but a singular braid with intersection points. Defining $f$ from $g_t$ via Eq.~\eqref{eq:ffromg} we obtain a semiholomorphic polynomial, whose singularity at the origin is not weakly isolated. By construction, $f$ is \textbf{radially weighted homogeneous}, that is, its support in $\mathbb{Z}^2$ lies on a straight line of negative slope. In particular, the constructed $f$ is equal to its principal part and is Newton degenerate \cite{eder2}.

We may then find an additional term $A(v,\bar{v})$ such that $f+A$ has a (weakly) isolated singularity, whose link is of the desired form, that is, all singular crossings of the singular braids formed by the roots of $g_t$ are resolved in a very controlled way \cite{bode:ak}. 

The important observation in the context of saddle point braids is that adding $A(v,\bar{v})$ does not change the derivative with respect to $u$, i.e., $f_u=(f+A)_u$.

\begin{proof}[Proof of Theorem~\ref{thm:semi}]
The proof is a modification of the construction in \cite{bode:ak}. First of all, recall that for every link $L$ and every sufficiently large integer $n$ there is a braid on $n$ strands that closes to $L$. By the same arguments as in \cite{bode:satellite} we can find for every sufficiently large even integer $n-1$ a braid $B$ on $n$ strands such that the closure of $B^2$ contains $L$ as a sublink. In particular, for every pair of links $L_1$ and $L_2$ and every sufficiently large even $n$ there exist braids $B_1$ on $n$ strands and $B_2$ on $n-1$ strands, so that $B_1$ closes to $L_1$ and the closure of $B_2^2$ contains $L_2$ as a sublink.

By Theorem~\ref{thm:saddle} there is a loop of polynomials $h_t$ in $\widehat{X}_n$ such that the roots of $h_t$ form the trivial braid and the saddle point braid of $h_t$ is $B_2$. Furthermore, after a homotopy of this loop, which does not change the braid types, we may assume that $h_0$ is a real polynomial, i.e., all of its roots are real numbers. As in \cite{bode:ak} we can also construct a loop of real polynomials $g_t$ whose roots form a singular braid $B_{sing}$ such that for every singular crossing $c$ there is a choice of crossing sign $\varepsilon_c\in\{\pm 1\}$ that turns $B_{sing}$ into a classical braid that is isotopic to $B_1$ if each singular crossing $c$ is replaced by a classical crossing of sign $\varepsilon_c$. We may take the basepoint of $g_t$ to be $h_0$.

Consider now the loop that is formed by the composition of $h_t$ and $g_t$. We may approximate its coefficients by polynomials in $\rme^{\rmi t}$ and $\rme^{-\rmi t}$. Call the resulting loop of polynomials $G_t$. The approximation can be chosen arbitrarily $C^1$-close and we can interpolate the original coefficient functions, so that the roots of $G_t$ form the singular braid $B_{sing}$ and the corresponding saddle point braid is $B_2$. Note that the saddle point braid of $g_t$ was the trivial braid, since it is a real polynomial whose $n$ roots have at most multiplicity 2.

Therefore, the roots of the loop $G_{2t}$ form the singular braid $B_{sing}^2$ and its saddle point braid is $B_2^2$. Furthermore, the loop obviously satisfies the symmetry constraint $G_{2(t+\pi)}=G_{2t}$, so that $f$ defined from $G_{2t}$ as in Eq.~\eqref{eq:ffromg} is a radially weighted homogeneous semiholomorphic polynomial.

As in \cite{bode:ak} we may now find an extra additive term $A(v,\bar{v})$, which gives $f+A$ a weakly isolated singularity at the origin and resolves all singular crossings in a way that guarantees that the resulting link is the closure of $B_1$, that is, $L_1$. The only difference to the proof in \cite{bode:ak} is that $G_{2t}$ is not a loop of real polynomials. However, recall that all singular crossings of the roots of $G_t$ occur in intervals where $G_{2t}$ is an arbitrarily close approximation of $g_t$, which is a loop of real polynomials. This is sufficient for the arguments from \cite{bode:ak}, so that $f+A$ has a weakly isolated singularity whose link is $L_1$. Furthermore, as observed above, we have $(f+A)_u=f_u$, so that the zeros of $(f+A)_u$ are exactly $\bigcup_{j=1}^n(r^kc_j(t),r \rme^{\rmi t})\subset\mathbb{C}$ with $r\geq 0$, $t\in[0,2\pi]$ and $\bigcup_{j=1}^n(c_j(t),t)$ a parametrisation of the saddle point braid of $G_{2t}$. Since the critical points of $G_{2t}$ are all distinct (they form a braid), this implies that all roots of $(f+A)_u$ except the origin are regular points of $(f+A)_u$. This means that $(f+A)_u$ has a weakly isolated singularity at the origin and its link is the closure of $B_2^2$, which by construction contains $L_2$ as a sublink.
\end{proof}

We see from the proof that in general the link of $(f+A)_u$ has extra components besides $L_2$. This is because by construction, the saddle point braid of $G_{2t}$ is a 2-periodic braid $B_2^2$. If on the on the other hand $L_2$ is the closure of a 2-periodic braid $B_2^2$ on $n-1$ strands, such that $n$ is at least the braid index of $L_1$, then the link of the singularity of $(f+A)_u$ is exactly $L_2$ without any extra components. 

\begin{proof}[Proof of Theorem~\ref{thm:thomosemi}]
The theorem follows almost immediately from the construction in \cite{bode:thomo}, where we start with a loop $g_t$, whose roots form a P-fibered geometric braid whose closure is the unknot and that satisfies $g_{t+\pi}(u)=-g_t(-u)$. We deform the critical values of $g_t$, so that there are some values of $t$ for which there is a critical value equal to 0, which means that at that value of $t$ the roots of the deformed $g_t$ are not disjoint and thus form a singular braid as $t$ varies from 0 to $2\pi$. It is proved in \cite{bode:thomo} that this deformation can be done in such a way that the resulting critical values $v_j(t)$ still satisfy the condition $\tfrac{\partial \arg(v_j(t))}{\partial t}\neq 0$ for all $j=1,2,\ldots,n-1$ and all values of $t$ with $v_j(t)\neq 0$. Furthermore, we may assume that the coefficients of the loop of polynomials $\widehat{g}_t$ that corresponds to the endpoint of the deformation are polynomials in $\rme^{\rmi t}$ and $\rme^{-\rmi t}$.

By Theorem~\ref{thm:thomosaddle} we can do this procedure, starting with a loop $g_t$ whose saddle point braid is the trivial braid on $n-1$ strands. Since the critical values are distinct throughout the deformation, the same is true for the critical points. It follows that the saddle point braid of $\widehat{g}_t$ is also the trivial braid on $n-1$ strands.

We may then construct $f$ from $\widehat{g}$ as in Eq.~\eqref{eq:ffromg}. By \cite{bode:thomo} there is a polynomial $A(v,\bar{v})$ such that $f+A$ has an isolated singularity at the origin and its link is the closure of the given T-homogeneous braid. Again, adding $A$ does not affect $f_u$, so that the roots of $(f+A)_u$ are precisely $(r^kc_j(t),r\rme^{\rmi t})\subset\mathbb{C}^2$, where $r\geq 0$, $t\in[0,2\pi]$ and $\bigcup_{j=1}^{n-1}(c_j(t),t)$ is a parametrisation of the saddle point braid of $\widehat{g}_t$, which is the trivial braid on $n-1$ strands. Thus $(f+A)_u$ has a weakly isolated singularity at the origin, whose link is the unlink on $n-1$ components.
\end{proof}

\section*{Acknowledgements}

B.B. is supported by the European Union's Horizon 2020 research and innovation programme through the Marie Sklodowska-Curie grant agreement 101023017. M.H. is supported by JSPS KAKENHI JP18K03296. We would like to thank Mark Dennis for helpful discussions.

%\textbf{Data Availability Statement}: Data sharing not applicable to this article as no datasets were generated or analysed during the current study.

\end{document}